\documentclass[graybox]{svmult}

\usepackage{graphicx}
\usepackage{mathptmx}       
\usepackage{helvet}         
\usepackage{courier}        
\usepackage{type1cm}        
\usepackage{amsmath}
\usepackage{amssymb}
\usepackage[all,2cell]{xy}\UseAllTwocells\SilentMatrices

\def\sinc{\mathop{\text{sinc}}}     

\def\image{\mathop{\text{image}}}
\def\ed{\mathop{\text{ed}}}   
\def\med{\mathop{\text{med}}}   

\def\shf{\mathcal}            
\def\col{\mathcal}            

\def\attach{\rightsquigarrow} 

\makeindex             

\begin{document}

\title*{A sheaf-theoretic perspective on sampling}
\author{Michael Robinson}
\institute{Michael Robinson \at American University, 4400 Massachusetts Ave NW, Washington, DC 20016, \email{michaelr@american.edu}}
%
%
\maketitle

\abstract*{Sampling theory has traditionally drawn tools from functional and complex analysis.  Past successes, such as the Shannon-Nyquist theorem and recent advances in frame theory, have relied heavily on the application of geometry and analysis.  The reliance on geometry and analysis means that the results are geometrically rigid.  There is a subtle interplay between the topology of the domain of the functions being sampled, and the class of functions themselves.  Bandlimited functions are somewhat limiting; often one wishes to sample from other classes of functions.  The correct topological tool for modeling all of these situations is the \emph{sheaf}; a tool which allows local structure and consistency to derive global inferences.  This chapter develops a general sampling theory for sheaves using the language of exact sequences, recovering the Shannon-Nyquist theorem as a special case.  It presents sheaf-theoretic approach by solving several different sampling problems involving non-bandlimited functions.  The solution to these problems show that the topology of the domain has a varying level of importance depending on the class of functions and the specific sampling question being studied.}

\section{Introduction}
Sampling theory has traditionally drawn tools from functional and complex analysis.  Past successes, such as the Shannon-Nyquist theorem and recent advances in frame theory, have relied heavily on the application of geometry and analysis.  The reliance on geometry and analysis means that the results are geometrically rigid.  

There is evidence that topology has an important -- and largely unexplored -- impact on sampling problems.  For instance, the space of bandlimited functions over the real line is infinite-dimensional, while the space of bandlimited functions over a compact subset is finite dimensional.  There is a subtle interplay between the topology of the domain of the functions being sampled, and the class of functions themselves.  For instance, one often wishes to sample from classes of non-bandlimited functions.  The correct algebraic tool for modeling all of these situations is the \emph{sheaf}; a tool which is sensitive to the topology of the domain and allows local structure to derive global inferences.  For instance, bandlimited functions over manifolds can be encoded in a sheaf.

Most sampling problems that have been studied in the literature assume that samples are scalar-valued and are collected uniformly in some fashion.  The theory of \emph{sheaf morphisms} formalizes and generalizes the sampling process, allowing each sample to be vector-valued and of different dimensions.

This chapter makes several contributions to sampling theory.  It proves a general sampling theorem for sheaves using the language of exact sequences.  The Shannon-Nyquist theorem is a special case of this more general sampling theorem, and we show how a sheaf-theoretic approach emphasizes the impact of topology by solving several different sampling problems involving non-bandlimited functions.  The solution to these problems shows that the topology of the domain has a varying level of importance depending on the class of functions and the specific sampling question being studied.

\subsection{Historical context}

Sampling theory has a long and storied history, about which a number of recent survey articles \cite{benedetto_1990,feichtinger_1994,unser_2000,smale_2004} have been written.  Since sampling plays an important role in applications, substantial effort has been expended on practical algorithms.  Our approach is topologically-motivated, like the somewhat different approach of \cite{NiySmaWeiHom,chazal_2009}, so it is less constrained by specific timing constraints.  Relaxed timing constraints are an important feature of bandpass \cite{vaughan_1991} and multirate \cite{unser_1998} algorithms.  We focus on signals with local control, of which splines \cite{unser_1999} are an excellent example.

Sheaf theory has not been used in applications until fairly recently.  The catalyst for new applications was the technical tool of \emph{cellular sheaves}, developed in \cite{Shepard_1985}.  Since that time, an applied sheaf theory literature has emerged, for instance \cite{ghrist_2011,Lilius_1993,GhristCurryRobinson,RobinsonQGTopo,RobinsonLogic}.

Our sheaf-theoretic approach allows sufficient generality to treat sampling on non-Euclidean spaces.  Others have studied sampling on non-Euclidean spaces, for instance general Hilbert spaces \cite{pesenson_2001}, Riemann surfaces \cite{schuster_2004}, symmetric spaces \cite{ebata_2006}, the hyperbolic plane \cite{feichtinger_2011}, combinatorial graphs \cite{pesenson_2010}, and quantum graphs \cite{pesenson_2005,pesenson_2006}.  We show that sheaves provide unified sufficiency conditions for perfect reconstruction on abstract simplicial complexes, which encompass all of the above cases.

A large class of local signals are those with \emph{finite rate of innovation} \cite{groechening_1992, vetterli_2002}.  Our ambiguity sheaf is a generalization of the Strang-Fix conditions as identified in \cite{dragotti_2007}.  With our approach, one can additionally consider reconstruction using richer samples than simply convolutions with a function.

\section{A unifying example}
\label{sec:unifying}
A celebrated consequence of the Cauchy integral formula is that the Taylor series of a holomorphic function evaluated at a point is sufficient to determine its value anywhere in its domain, if this is connected.  Analytic continuation is therefore a very strong kind of reconstruction from a single sample.  Analytic continuation relies both on (1) a restricted space of functions (merely smooth functions do not suffice) and (2) a rather large amount of information at the sample point (not just the value of the function, but also all of its derivatives).  These two constraints are essential to understand the nature of reconstruction from samples, so the admittedly special case of analytic continuation is informative. 

Consider the space of holomorphic functions $C^\omega(U,\mathbb{C})$ on a connected open set $U\subseteq \mathbb{C}$.  Without loss of generality, suppose that $U$ contains the origin.  Then the function $a:C^\omega(U,\mathbb{C}) \to l^1$ given by
\begin{equation*}
a(f) = \left(f(0), f'(0), \dotsc, \frac{f^{(n)}(0)}{n!}, \dotsc \right)
\end{equation*}
for $f\in C^\omega(U,\mathbb{C})$ is a linear transformation.  Because $a$ computes the Taylor series of $f$, whenever $a(f)=a(g)$ it must follow that $f=g$ on $U$.  This means that as a linear transformation, the sampling function $a$ has a trivial kernel.  

Conversely, the trivial kernel of $a$ witnesses the fact that the original $f \in C^\omega(U,\mathbb{C})$ can be recovered from the sampled value $a(f)$.  This is by no means necessarily true for all sampling functions. For instance, the sampling function $b:C^\omega(U,\mathbb{C}) \to  l^1$ given by
\begin{equation*}
b(f) = \left(0, f'(0), \dotsc, \frac{f^{(n)}(0)}{n!}, \dotsc \right)
\end{equation*}
has a one-dimensional kernel.  This means that reconstruction of an analytic function from its image through $b$ is ambiguous -- it is known only up to the addition of a constant.  But there is more information available than merely the \emph{dimension} of $\ker b$, since it is a subspace of $C^\omega(U,\mathbb{C})$.  Indeed, if we restrict the domain of $b$ to be the subspace $Z \subset C^\omega(U,\mathbb{C})$ of analytic functions whose value at the origin is zero, then the intersection $Z \cap \ker b$ is trivial.  Reconstruction succeeds on $Z$ using $b$ for sampling even though using $b$ on its whole domain is ambiguous.  

Sampling a function in $C^\omega(U,\mathbb{C})$ can be represented generally as a function $s:C^\omega(U,\mathbb{C}) \to l^1$.  Observe that $s$ could take the form of the functions $a$ or $b$ above, in which a function is evaluated in the immediate vicinity of a single point.  However, $s$ could also be given by 
\begin{equation*}
s(f) = (\dotsc, f(-1), f(0), f(1), \dotsc)
\end{equation*}
or many other possibilities.  In this general setting, the simplest way to determine whether reconstruction is ambiguous is to examine $\ker s$.  Recognizing that we may wish to restrict the class of functions under discussion, it is useful to understand how the subspace $\ker s$ is included within $C^\omega(U,\mathbb{C})$.  This situation can also be described as the following \emph{exact sequence} of linear functions
\begin{equation*}
\xymatrix{
0 \to A \ar[r]^i & C^\omega(U,\mathbb{C}) \ar[r]^s & l^1 \to 0,
}
\end{equation*}
which means that $\ker s = \image i$.  Observe that the zero at the beginning of the sequence indicates that the map $i$ is injective, so that $A = \ker s$.  Likewise, the zero at the end of the sequence indicates that $s$ is surjective.  

If $f \in C^\omega(U,\mathbb{C})$ is fixed (but unknown) and $s(f)$ is known, then clearly $f$ can only be known to be one of the preimages $i^{-1}(f) \subseteq A$.  If there is only one preimage (as in the case of $a$ above), then reconstruction is said to be \emph{unambiguous}.  

This example contrasts sharply with the situation of sampling data from the space of \emph{all} smooth functions.  In this case, one has a diagram like
\begin{equation*}
\xymatrix{
0\to C \ar[r]^i & C^\infty(U,\mathbb{C}) \ar[r]^c & l^1 \to  0,
}
\end{equation*}
because in this case $C$ is quite large.  The analytic functions are a subset $C^\omega(U,\mathbb{C}) \subset C^\infty(U,\mathbb{C})$.  This can be expressed diagrammatically as
\begin{equation*}
\xymatrix{
       &   & 0 \ar[d]                      & 0 \ar[d] & \\
&0 \ar[r]^i\ar[d] & C^\omega(U,\mathbb{C}) \ar[r]^a\ar[d] & l^1 \ar[r]\ar[d]& 0 \\
0\ar[r]&C \ar[r]^i & C^\infty(U,\mathbb{C}) \ar[r]^c & l^1\ar[r]& 0\\
}
\end{equation*}
in which every possible composition of linear maps with the same domain and codomain are equal.  This shows how the two classes of functions and their samples are related, and the technique will be used in later sections as a kind of algebaic bound.

\section{Local data}

Vector spaces of functions such as $C^k(U,\mathbb{C})$ are rather global in nature -- an element of such a space \emph{is} a function!  In contrast, evaluating a function at a particular point $x$ corresponds to a linear transformation that is only sensitive to a function's value at or nearby $x$.  Because function evaluation is a local process, reconstructing the global function space element from these samples appears counterintuitive.  

The local sampling versus global reconstruction paradox is resolved because reconstruction theorems only exist for certain suitably constrained vector spaces.  For instance, the Paley-Wiener space $PW_B$ consists of functions $f$ whose Fourier transform
\begin{equation*}
\hat{f}(\omega) = \int_{-\infty}^{\infty} f(x) e^{-2\pi i \omega x} dx
\end{equation*}
is supported on $[-B,B]$.  We say that each $f\in PW_B$ has \emph{bandwidth $B$}.  The Shannon-Nyquist theorem asserts that functions in $PW_{1/2}$ are uniquely determined by their values on the integers, which is best explained by the fact that every $f\in PW_{1/2}$ has a cardinal series decomposition
\begin{equation*}
f(x) = \sum_{n=-\infty}^\infty f(n) \frac{\sin\pi(x-n)}{\pi(x-n)}
\end{equation*}
where each \emph{sinc function} given by $\sinc(x-n)=\frac{\sin(x-n)}{x-n}$ has bandwidth $1/2$ or less.  Moreover, the set of sinc functions is orthonormal over the usual inner product in $PW_{1/2}$, so we have that
\begin{equation}
\label{eq:int_samp_sinc}
f(x) = \int_{n=-\infty}^\infty f(n) \frac{\sin\pi(x-n)}{\pi(x-n)}\, dx.
\end{equation}
Even though the support of $\sinc \pi(x-n)$ is $\mathbb{R}$, it decays away from $n$.  This means that in \eqref{eq:int_samp_sinc}, the effect of values of $f$ far away from $n$ will have little effect on $f(n)$.  So in the case of $PW_B$, sampling via \eqref{eq:int_samp_sinc} is only \emph{approximately} local.  Because of this, global constraints -- such as those arising from compactness -- on the function space play an important role in sampling theorems.  

This section formalizes the above intuition, by constructing a \emph{sheaf theoretic} framework for discussing sampling.  Sheaves are the correct mathematical formalism for discussing local information.  Section \ref{sec:defs} distills an axiomatic framework that precisely characterizes what ``local'' means.  Section \ref{sec:cohomology} defines the cohomology functor for sheaves, which assembles this local information into global information.  With the definition of a \emph{sheaf morphism} in Section \ref{sec:transformations}, these tools allow the statement of general conditions under which a sampling suffices to reconstruct a function in a particular space in Section \ref{sec:general_sampling}.  

\subsection{Sheaves represent local data}
\label{sec:defs}

A local model of data should be flexible enough to capture both analytic and non-analytic functions.  Because portions of the data in one region will not necessarily be related to those farther away, the model should allow us to infer global effects only when they are appropriate to the kind of function under study.  

Spaces of continuous functions exhibit several properties related to locality.  As a concrete example, consider the following properties of $C^k(U,\mathbb{C})$ when $k \ge 0$:
\begin{enumerate}
\item \emph{Restriction}: Whenever $V \subseteq U$ are open sets, there is a linear map $C^k(U,\mathbb{C}) \to C^k(V,\mathbb{C})$ that is given by restricting the domain of a function defined on $U$ to one defined on $V$.  
\item \emph{Uniqueness}:  Whenever a function is the zero function on some open set, then all of its restrictions are zero functions also.  The converse is true also: suppose $f \in C^k(V,\mathbb{C})$, and that $\{U_1,\dotsc\}$ is an open cover of $V$.  If the restriction of function $f$ to each $U_k$ is the zero function on $U_k$, then $f$ has to be the zero function on $V$.
\item \emph{Gluing}: If $U$ and $V$ are open sets and $f \in C^k(U,\mathbb{C})$, $g \in C^k(V,\mathbb{C})$ then whenever $f(x)=g(x)$ for all $x\in U \cap V$ there is a function $h \in C^k(U \cup V,\mathbb{C})$ that restricts to $f$ and $g$.
\end{enumerate}

The gluing property provides a condition by which local information (the elements $f\in C^k(U,\mathbb{C})$, $g\in C^k(V,\mathbb{C})$) can be assembled into global information in $C^k(U\cup V,\mathbb{C})$, provided a consistency condition is met.  We will call this specification of $f$ and $g$ a \emph{section} when they restrict to the same element in $C^k(U \cap V,\mathbb{C})$.  This can also be illustrated diagrammatically
\begin{equation*}
\xymatrix{
&C^k(U \cup V,\mathbb{C})\ar[dr]\ar[dl]&\\
C^k(U,\mathbb{C})\ar[dr]&&C^k(V,\mathbb{C})\ar[dl]\\
&C^k(U\cap V,\mathbb{C})&\\
}
\end{equation*}
where the arrows represent the restrictions of functions from one domain to the next.  Specifically, when two functions on the middle level are mapped to the same function on the bottom level, they are both images of a function on the top level.

Let's formalize these properties to obtain a more general construction, in which the data are not necessarily encoded as continuous functions.  It is usually unnecessary to consider \emph{all} open sets; what's really relevant is the intersection lattice.  In this chapter, we need a concept of space that is convenient for computations.  The most efficient such definition is that of a simplicial complex.  

\begin{definition}
An \emph{abstract simplicial complex} $X$ on a set $A$ is a collection of ordered subsets of $A$ that is closed under the operation of taking subsets.  We call each element of $X$ a \emph{face}.  A face with $k+1$ elements is called a $k$-dimensional face (or a $k$-face), though we usually call a $0$-face a \emph{vertex} and a $1$-face an \emph{edge}.  If all of the faces of an abstract simplicial complex $X$ are of dimension $n$ or less, we say that $X$ is an $n$-dimensional simplicial complex.  If $X$ is a 1-dimensional simplicial complex, we usually call $X$ a \emph{graph}.  

The \emph{face category} has the elements of $X$ for its objects, and setwise inclusions of one element of $X$ into another for its morphisms.  If $a$ and $b$ are two faces in an abstract simplicial complex $X$ with $a \subset b$ and $|a| < |b|$, we will write $a \attach b$ and say that $a$ is \emph{attached} to $b$.  Finally, a collection $Y$ of faces of $X$ is called a \emph{closed subcomplex} if whenever $b \in Y$ and $a \attach b$, then $a\in Y$ also.
\end{definition}

Sometimes simplicial complexes arise naturally from the problem, for instance the connection graph for a network, but it is helpful to have a procedure to obtain a simplicial complex from a topological space.  Suppose that $X$ is a topological space and that $\col{U} = \{U_1,\dotsc\}$ is an open cover of $X$.
\begin{definition}
The \emph{nerve} $N(\col{U})$ is the abstract simplicial complex whose vertices are given by the elements of $\col{U}$, and whose $k$-faces $\{U_{i_0},\dotsc,U_{i_k}\}$ are given by the nonempty intersections $U_{i_0} \cap \dotsb \cap U_{i_k}$.
\end{definition}

\begin{figure}
\begin{center}
\includegraphics[width=3in]{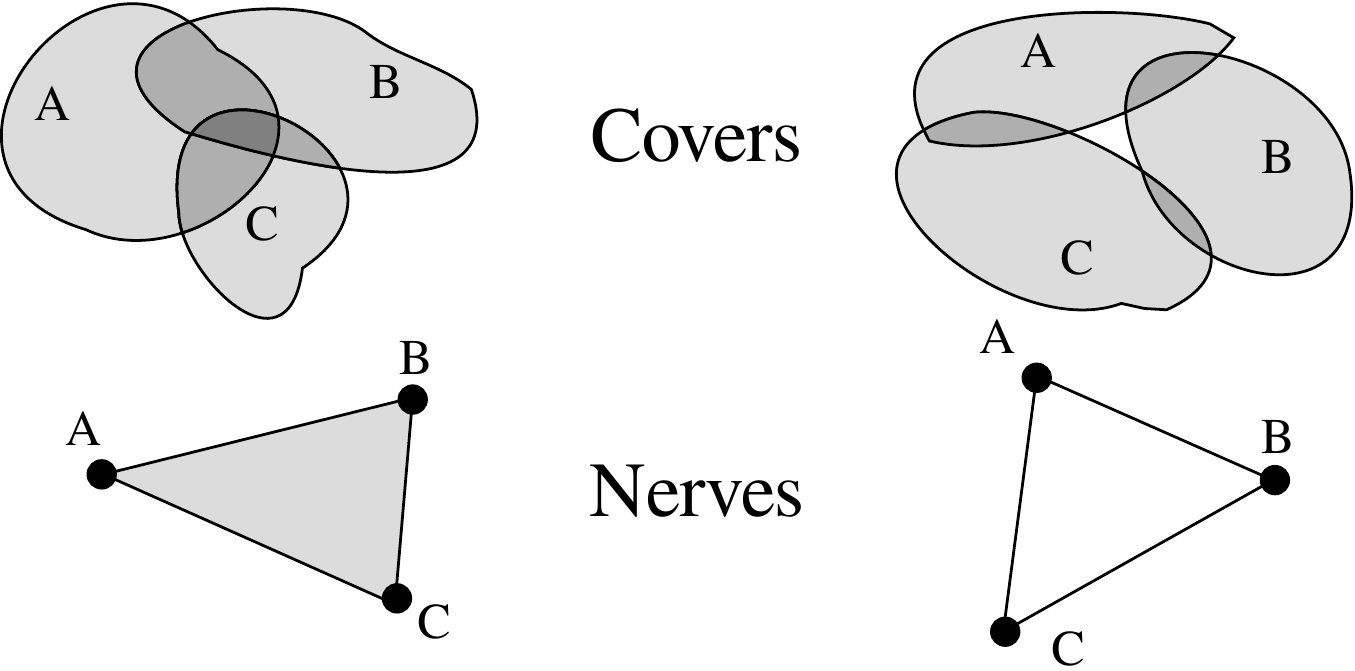}
\caption{The nerve of two covers: (left) with a nonempty triple intersection (right) without a triple intersection}
\label{fig:nerve}
\end{center}
\end{figure}

\begin{example}
\label{eg:nerve}
Figure \ref{fig:nerve} shows two covers and their associated nerves.  In the left diagram, the sets $A$, $B$, and $C$ have nonempty pairwise intersections and a nonempty triple intersection $A \cap B \cap C$, so the nerve is a 2-dimensional abstract simplicial complex.  In the right diagram, $A \cap B \cap C$ is empty, so the nerve is only 1-dimensional.
\end{example}

The concept of local information over a simplicial complex is a straightforward generalization of the three properties (restriction, uniqueness, and gluing) for continuous functions.  The resulting mathematical object is called a \emph{sheaf}.

\begin{definition}
\label{df:sheaf}
A \emph{sheaf} $\shf{F}$ on an abstract simplicial complex $X$ is a covariant functor from the face category of $X$ to the category of vector spaces.  Explicitly, 
\begin{itemize}
\item for each element $a$ of $X$, $\shf{F}(a)$ is a vector space, called the \emph{stalk at $a$},
\item for each attachment of two faces $a\attach b$ of $X$, $\shf{F}(a\attach b)$ is a linear function from $\shf{F}(a)\to \shf{F}(b)$ called a \emph{restriction}, and 
\item for every composition of attachments $a\attach b \attach c$, the restrictions satisfy $\shf{F}(b\attach c) \circ \shf{F}(a\attach b)=\shf{F}(a\attach b\attach c)$.
\end{itemize}
We will usually refer to $X$ as the \emph{base space} for $\shf{F}$.
\end{definition}

\begin{remark}
Although sheaves have been extensively studied over topological spaces (see \cite{Bredon} or the appendix of \cite{Hubbard} for a modern, standard treatment), the resulting definition is ill-suited for application to sampling.  Instead, we follow a substantially more combinatorial approach introduced in the 1980 thesis of Shepard \cite{Shepard_1980}.  
\end{remark}

\begin{figure}
\begin{center}
\includegraphics[width=3.25in]{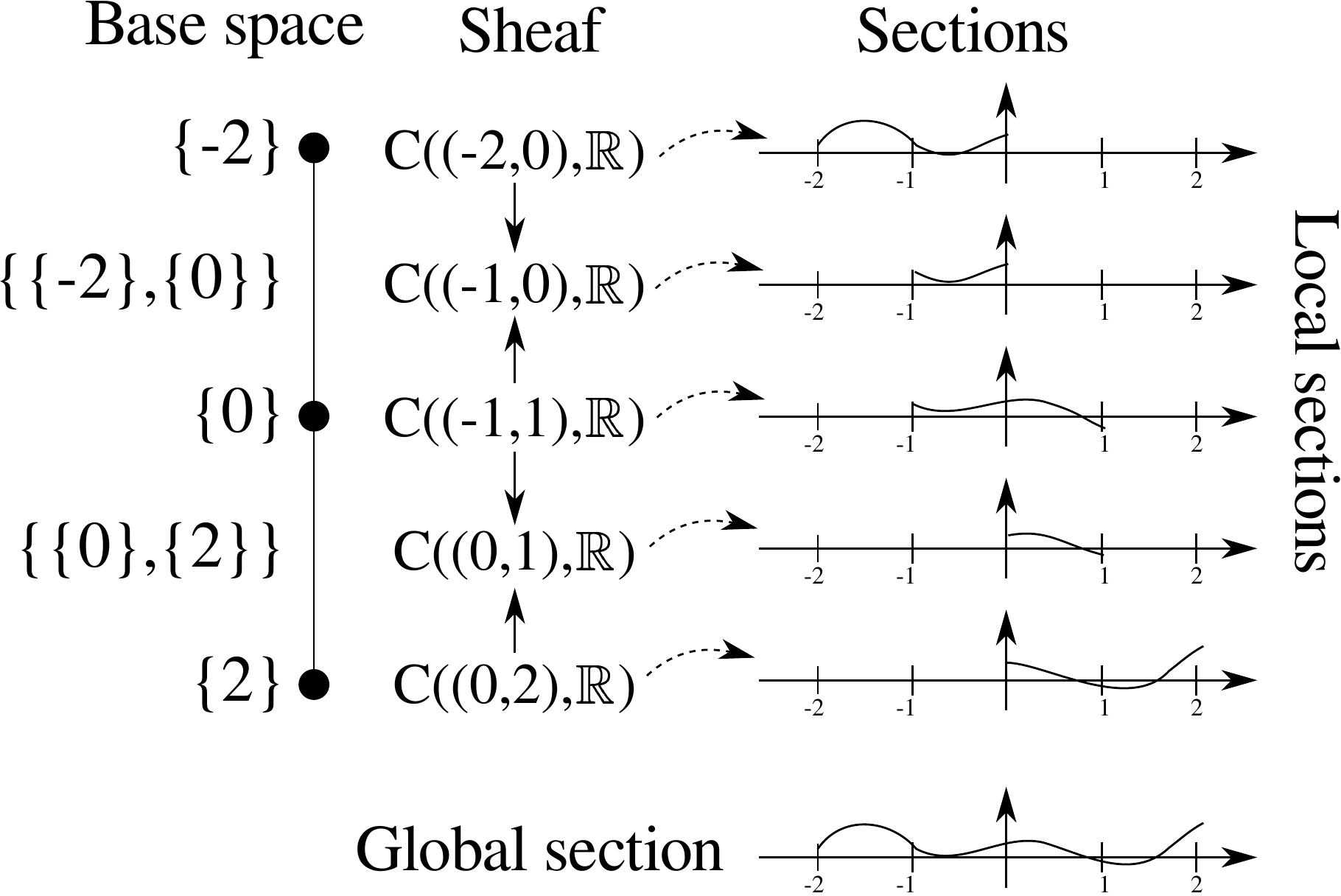}
\caption{A sheaf of continuous functions over an interval (compare with Figure \ref{fig:continuous_sheaf2})}
\label{fig:continuous_sheaf1}
\end{center}
\end{figure}

\begin{example}
\label{eg:continuous_sheaf1}
The space of continuous functions over a topological space can be represented as a sheaf.  For instance, Figure \ref{fig:continuous_sheaf1} shows one way to organize the space of continuous functions over the interval $(-2,2)$ in terms of spaces of continuous functions over smaller intervals.  (See Example \ref{eg:continuous_sheaf2} for another encoding of continuous functions as a sheaf.)  In this particular sheaf model, the base space is given by an abstract simplicial complex $X$ over three abstract vertices, which we label suggestively as $\{-2\}$, $\{0\}$, $\{2\}$.  The simplicial complex $X$ has two edges as shown in the diagram.  We define the sheaf $\shf{C}$ over $X$ by assigning spaces of continuous functions to each face, and define the restrictions between stalks to be the process of ``actually'' restricting the functions.
\end{example}

\begin{example}
\label{eg:analytic_sheaf}
Coming back to Section \ref{sec:unifying}, a sheaf of analytic functions can be constructed as a subsheaf of the previous example by merely replacing the stalks with spaces of analytic functions defined over the appropriate intervals.
\end{example}

Notice that the definition of a sheaf captures the \emph{restriction} locality property, but does not formalize the \emph{uniqueness} or \emph{gluing} properties.  Some authors \cite{Bredon, Curry, Iverson, Godement_1958} explicitly require these properties from the outset, calling the object defined in Definition \ref{df:sheaf} a \emph{presheaf}, regarding it as incomplete.  Although the difference between sheaves and presheaves is useful in navigating certain technical arguments, every presheaf has a unique \emph{sheafification}.  Because of this, our strategy follows the somewhat more economical treatment set forth in \cite{Shepard_1980, Robinson_tspbook}, which removes this distinction.  As a consequence of this choice, we explicitly define collections of \emph{sections}, which effectively implement the sheafification.

\begin{definition}
\label{df:section}
Suppose $\shf{F}$ is a sheaf on an abstract simplicial complex $X$ and that $\col{U}$ is a collection of faces of $X$.  An assignment $s$ which assigns an element of $\shf{F}(a)$ to each face $a\in\col{U}$ is called a \emph{section} supported on $\col{U}$ when for each attachment $a\attach b$ of faces in $\mathcal{U}$, $\shf{F}(a\attach b)s(a)=s(b)$.  We will denote the space of sections of $\shf{F}$ over $\col{U}$ by $\shf{F}(\col{U})$, which is easily checked to be a vector space.  A \emph{global section} is a section supported on $X$.  If $r$ and $s$ are sections supported on $\col{U} \subset \col{V}$, respectively, in which $r(a)=s(a)$ for each $a \in \col{U}$ we say that \emph{$s$ extends $r$}.  
\end{definition}

\begin{example}
\label{eg:sample_sheaf}
Consider $Y\subseteq X$ a subset of the vertices of an abstract simplicial complex. The sheaf $\shf{S}$ which assigns a vector space $V$ to vertices in $Y$ and the trivial vector space to every other face is called a \emph{$V$-sampling sheaf supported on $Y$}.  To every attachment of faces of different dimension, $S$ will assign the zero function.   For a finite abstract simplicial complex $X$, the space of global sections of a $V$-sampling sheaf supported on $Y$ is isomorphic to $\bigoplus_{y\in Y} V$.
\end{example}

\begin{figure}
\begin{center}
\includegraphics[width=3.25in]{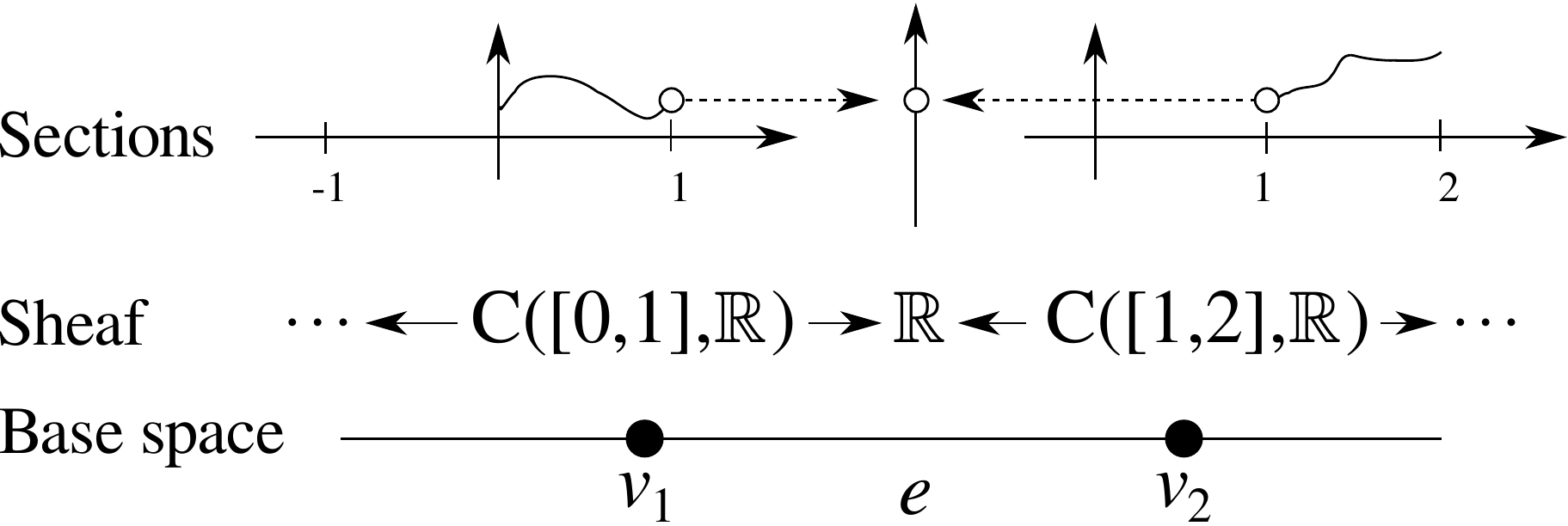}
\caption{Another sheaf of continuous functions over an interval (compare with Figure \ref{fig:continuous_sheaf1})}
\label{fig:continuous_sheaf2}
\end{center}
\end{figure}

\begin{example}
\label{eg:continuous_sheaf2}
Figure \ref{eg:continuous_sheaf2} shows a sheaf whose global sections are continuous functions that is essentially dual to the one in Example \ref{eg:continuous_sheaf1}.  (Although it is straightforward to generalize the construction to cell complexes of arbitrary dimension, we will work over an interval to keep the exposition simple.)  Specifically, consider a simplicial complex with two vertices $v_1$ and $v_2$ and one edge $e$ between them.  The stalk over each vertex is a space of continuous functions as in Example \ref{eg:continuous_sheaf1}, though we require the functions to be continuous over a \emph{closed} interval.  However, the stalk over the edge is merely $\mathbb{R}$.  The restriction in this case evaluates functions at an appropriate endpoint.  If we name the sheaf $\shf{C}$, then for $f \in \shf{C}(v_1)=C([0,1],\mathbb{R})$,
\begin{equation*}
\left(\shf{C}(v_1\attach e)\right)(f)=f(1),
\end{equation*}
and
\begin{equation*}
\left(\shf{C}(v_2\attach e)\right)(g)=g(1),
\end{equation*}
for $g \in \shf{C}(v_2)=C([1,2],\mathbb{R})$.  Observe that the global sections of this sheaf are precisely functions that are continuous on $[0,2]$. 
\end{example}

This second example of a sheaf of continuous function can be easily recast to describe discrete timeseries as well.  However, it is convenient to discriminate between the restrictions ``to the left'' versus those ``to the right.''  This is most conveniently described by the concept of orientation.  Recall that an abstract simplicial complex $X$ consists of \emph{ordered} sets.  For $a$ a $k$-face and $b$ a $k+1$-face, define the \emph{orientation index}
\begin{equation*}
[b:a]=
\begin{cases}
+1&\text{if the order of elements in }a\text{ and }b\text{ agrees,}\\
-1&\text{if it disagrees, or}\\
0&\text{if }a\text{ is not a face of }b.\\
\end{cases}
\end{equation*}

\begin{example}
\label{eg:grouping_sheaf}
If $V$ is a vector space, then the $m$-term grouping sheaf $\shf{V}^{(m)}$ has the diagram
\begin{equation*}
\xymatrix{
\ar[r]^{\sigma_+} & V^{m-1} & V^m \ar[r]^{\sigma_+} \ar[l]^{\sigma_-} & V^{m-1} & \ar[l]^{\sigma_-} 
}
\end{equation*}
in which the restrictions are given by
\begin{equation*}
\sigma_-(x_1,\dotsc,x_m)=(x_1,\dotsc,x_{m-1})\text{ and }\sigma_+(x_1,\dotsc,x_m)=(x_2,\dotsc,x_m).
\end{equation*}
Observe that the sampling sheaf defined in Example \ref{eg:sample_sheaf} is merely $V^{(1)}$, and that the space of global sections of all grouping sheaves are isomorphic.
\end{example}

Sheaves can also describe spaces of piecewise continuous functions, as the next example shows.

\begin{figure}
\begin{center}
\includegraphics[width=3in]{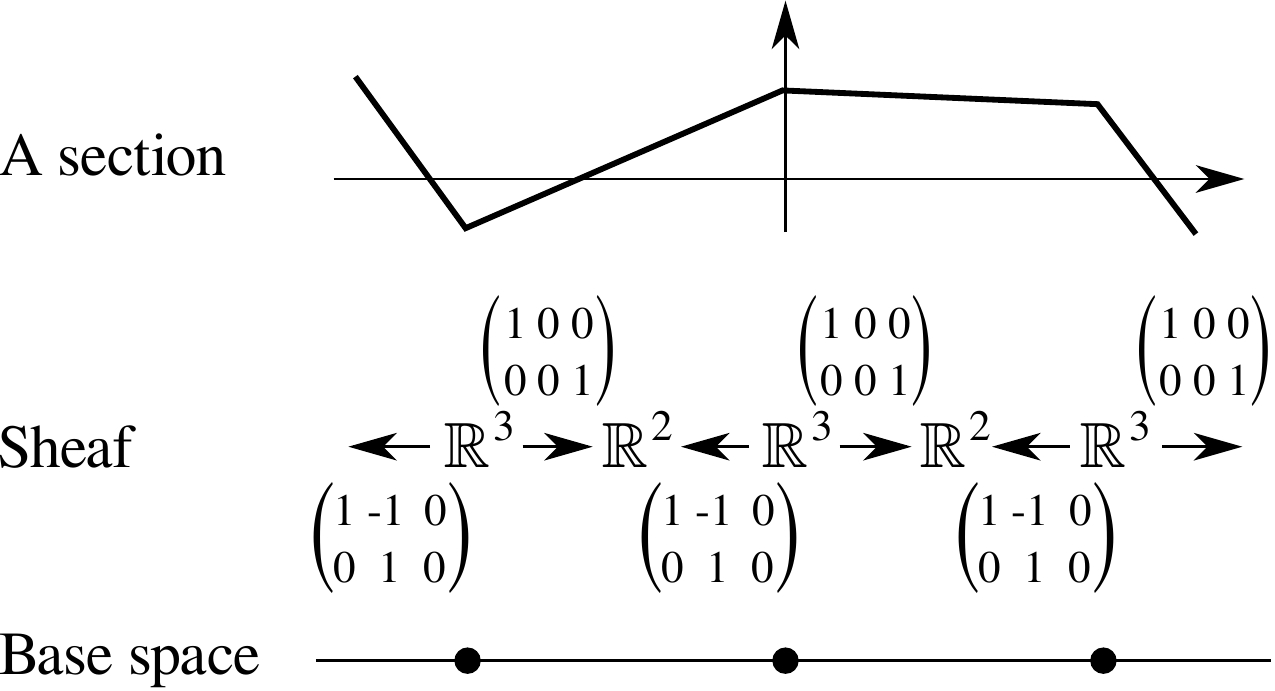}
\caption{An example of a sheaf $\shf{PL}$ over a graph}
\label{fig:pl_sheaf}
\end{center}
\end{figure}

\begin{example}
\label{eg:pl_def}
Suppose $G$ is a graph in which each vertex has finite degree.  Let $\shf{PL}$ be the sheaf constructed on $G$ that assigns $\shf{PL}(v)=\mathbb{R}^{1+\deg v}$ to each edge $v$ of degree $\deg v$ and $\shf{PL}(e)=\mathbb{R}^2$ to each edge $e$.  The stalks of $\shf{PL}$ specify the value of the function (denoted $y$ below) at each face and the slopes of the function on the edges (denoted $m_1,...,m_k$ below).  To each attachment of a degree $k$ vertex $v$ into an edge $e$, let $\shf{PL}$ assign the linear function 
\begin{equation*}
\left(\shf{PL}(v\attach e)\right)(y,m_1,...,m_e,...,m_k)=
\begin{pmatrix}
y+([e:v]-1)\frac{1}{2}m_e\\
m_e
\end{pmatrix}.
\end{equation*}
The global sections of this sheaf are \emph{piecewise linear functions} on $G$; see Figure \ref{fig:pl_sheaf}.  
\end{example}

\subsection{Sheaf cohomology}
\label{sec:cohomology}

\begin{figure}
\begin{center}
\includegraphics[width=1.25in]{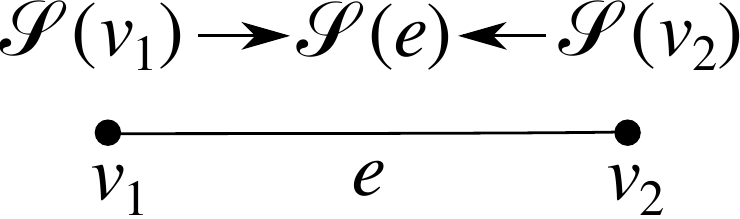}
\caption{A sheaf over a small abstract simplicial complex}
\label{fig:coho_moti}
\end{center}
\end{figure}

The space of global sections of a sheaf is important in applications.  Although Definition \ref{df:section} is not constructive, one can compute this space algorithmically.  Specifically, consider the abstract simplicial complex $X$ shown in Figure \ref{fig:coho_moti}, which consists of an edge $e$ between two vertices $v_1$ and $v_2$.  Suppose $s$ is a global section of a sheaf $\shf{S}$ on $X$.  This means that
\begin{equation*}
\shf{S}(v_1\attach e)s(v_1) = s(e) = \shf{S}(v_2 \attach e)s(v_2).
\end{equation*}
Since the above equation is written in a vector space, we can rearrange it to obtain the equivalent specification
\begin{equation*}
\shf{S}(v_1\attach e)s(v_1) - \shf{S}(v_2 \attach e)s(v_2) = 0,
\end{equation*}
which could be written in matrix form as
\begin{equation*}
\begin{pmatrix}
\shf{S}(v_1\attach e) & -\shf{S}(v_2\attach e)
\end{pmatrix}
\begin{pmatrix}
s(v_1)\\ s(v_2)
\end{pmatrix}
=0.
\end{equation*}
This purely algebraic manipulation shows that computing the space of global sections of a sheaf is equivalent to computing the kernel of a particular matrix as in Section \ref{sec:unifying}.  Clearly this procedure ought to work for arbitrary sheaves over arbitrary abstract simplicial complexes, though it could get quite complicated.  \emph{Cohomology} is a systematic way to perform this computation, and it results in additional information as we'll see in later sections.  

The vector $\begin{pmatrix}
s(v_1)\\ s(v_2)
\end{pmatrix}$ above suggests that we should define following formal \emph{cochain} vector spaces $C^k(X;\shf{F})=\bigoplus_{a\text{ a }k\text{-face of }X}\shf{F}(a)$ to represent the possible choices of data over the $k$-faces.  In the same way, the matrix 
\begin{equation*}
\begin{pmatrix}
\shf{S}(v_1\attach e) & -\shf{S}(v_2\attach e)
\end{pmatrix}
\end{equation*}
generalizes into the \emph{coboundary map} $d^k:C^k(X;\shf{F})\to C^{k+1}(X;\shf{F})$, which we now define.  The coboundary map $d^k$ takes an assignment $s$ on the $k$-faces to a different assignment $d^ks$ whose value at a $(k+1)$-face $b$ is
\begin{equation*}
(d^ks)(b)=\sum_{a\text{ a }k\text{-face of} X} [b:a]\shf{F}(a\attach b)s(a).
\end{equation*}
Together, we have a sequence of linear maps
\begin{equation*}
\xymatrix{
0\to C^0(X;\shf{F}) \ar[r]^{d^0} & C^1(X;\shf{F}) \ar[r]^{d^1} & C^2(X;\shf{F}) \ar[r]^{d^2} & \dotsb
}
\end{equation*}
called the \emph{cochain complex}.

As in the simple example described above, the kernel of $d^k$ consists of data specified on $k$-faces that is consistent, when tested on the $(k+1)$-faces.  However, it can be shown that $d^{k+1}\circ d^k=0$, so that the image of $d^k$ is a subspace of the kernel of $d^{k+1}$.  This means that the image of $d^k$ is essentially redundant information, since it is already known to be consistent when tested on the $(k+2)$-faces.  Because of this fact, only those elements of the kernel of $d^k$ that are \emph{not} already known to be consistent are really worth mentioning.  This leads to the definition of sheaf cohomology:   

\begin{definition}
The $k$-th \emph{sheaf cohomology} of $\shf{F}$ on an abstract simplicial complex $X$ is 
\begin{equation*}
H^k(X;\shf{F})=\ker d^k / \image d^{k-1}.
\end{equation*}
\end{definition}

As an immediate consequence of this construction, we have the following useful statement.

\begin{proposition}
\label{prop:global_coho}
$H^0(X;\shf{F})=\ker d^0$ consists precisely of those assignments $s$ which are global sections, so a global section is determined entirely by its values on the vertices of $X$.  
\end{proposition}

\subsection{Transformations of local data}
\label{sec:transformations}
Sheaves can be used to represent local data and cohomology can be used to infer the resulting globally-consistent data.  However interesting this theory may be, we need to connect it to the process of sampling.  Indeed, as envisioned in Section \ref{sec:unifying}, sampling is a \emph{transformation} between two spaces of functions -- from functions with a continuous domain to functions with a discrete domain.  Such a transformation arising from sampling respects the local structure of the function spaces.  This kind of transformation is called a \emph{sheaf morphism}.  There are two aspects to a sheaf morphism: (1) its effect on the base space, and (2) its effect on stalks.  The effect on the base space should be to respect local neighborhoods, which means that a sheaf morphism must at least specify a continuous map.  Since we have restricted our attention to abstract simplicial complexes rather than general topological spaces, the analog of a continuous map is a simplicial map.

\begin{definition}
A \emph{simplicial map} from one abstract simplicial complex $X$ to another $Y$ is a function $f$ from the set of simplices of $X$ to the simplices of $Y$ that additionally satisfies two properties:
\begin{enumerate}
\item If $a \attach b$ is an attachment of two simplices in $X$, then $f(a) \attach f(b)$ is an attachment of simplices in $Y$, and
\item The dimension of $f(a)$ is no more than the dimension of $a$, a simplex in $X$.
\end{enumerate}
\end{definition}

The last condition means a simplicial map takes vertices to vertices, edges either to edges or vertices, and so on. 

\begin{figure}
\begin{center}
\includegraphics[width=2.5in]{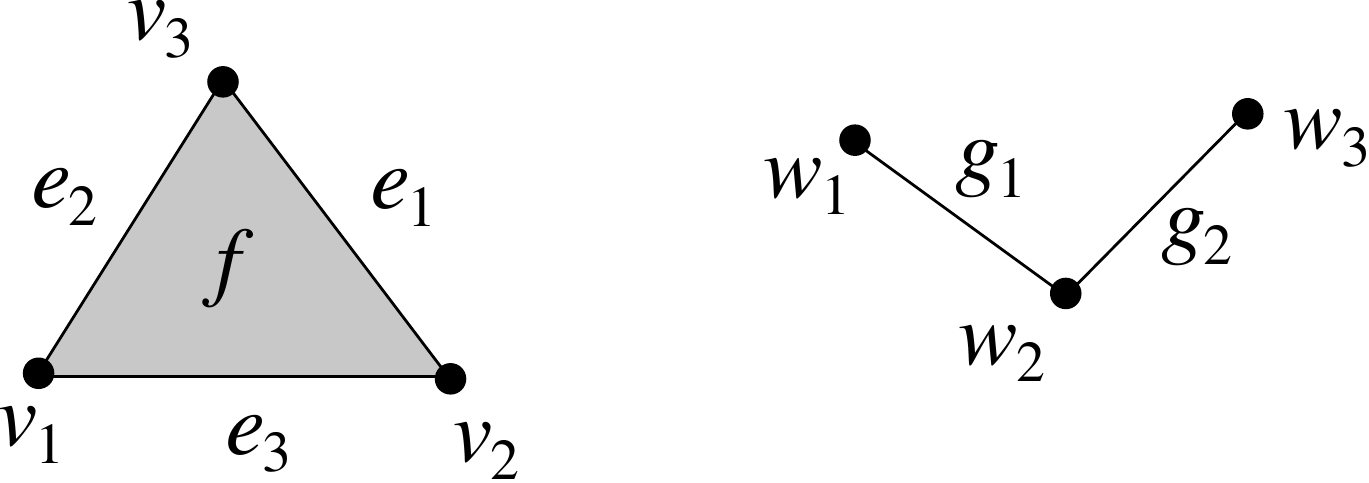}
\caption{The simplicial complexes $X$ (left) and $Y$ (right) for Example \ref{eg:simp_map}}
\label{fig:simp_map}
\end{center}
\end{figure}

\begin{example}
\label{eg:simp_map}
Consider the simplicial complexes $X$ and $Y$ shown in Figure \ref{fig:simp_map}.  The function $F:X\to Y$ given by 
\begin{equation*}
F(v_1)=w_1,\;
F(v_2)=w_2,\;
F(v_3)=w_2
\end{equation*}
determines a simplicial map, in which $F(e_1)=w_2$, $F(e_2)=F(e_3)=F(f)=g_1$.

In contrast, any function that takes $v_1$ to $w_1$, $v_2$ to $w_2$, and $v_3$ to $w_3$ cannot be a simplicial map because the image of $e_2$ should be an edge from $w_1$ to $w_3$, but no such edge exists.
\end{example}

\begin{definition}
Suppose that $f:X\to Y$ is a simplicial map, and that $\shf{F}$ is a sheaf on $Y$ and $\shf{G}$ is a sheaf on $X$.  A \emph{sheaf morphism} (or simply a \emph{morphism}) $m:\shf{F}\to \shf{G}$ along $f$ assigns a linear map $m_a:\shf{F}(f(a)) \to \shf{G}(a)$ to each face $a \in X$ so that for every attachment $a\attach b$ in the face category of $X$, $m_b \circ \shf{F}(f(a)\attach f(b)) = \shf{G}(a\attach b) \circ m_a$.
\end{definition}

Usually, we describe a morphism by way of a commutative diagram like the one below
\begin{equation*}
\xymatrix{
\shf{F}(f(a))\ar[d]_{\shf{F}(f(a)\attach f(b))}\ar[r]^{m_a} & \shf{G}(a) \ar[d]^{\shf{G}(a\attach b)}\\
\shf{F}(f(b))\ar[r]^{m_b} & \shf{G}(b)\\
}
\end{equation*}

\begin{remark}
The reader is cautioned that a sheaf morphism and its underlying simplicial map ``go opposite ways.'' 
\end{remark}

Cohomology is a functor from the category of sheaves and sheaf morphisms to the category of vector spaces.  This indicates that cohomology preserves and reflects the underlying relationships between data stored in sheaves.

\begin{proposition}
\label{prop:coho_functor}
Suppose that $\shf{R}$ is a sheaf on $X$ and that $\shf{S}$ is a sheaf on $Y$.  If $m:\shf{R}\to\shf{S}$ is a morphism of these sheaves, then $m$ induces linear maps $m^k:H^k(X,\shf{R})\to H^k(Y,\shf{S})$ for each $k$.  (Note that the simplicial map associated to $m$ is a function $Y\to X$.)
\end{proposition}

As a consequence, $m^0$ is a linear map from the space of global sections of $\shf{R}$ to the space of global sections of $\shf{S}$.  Because of this, it is possible to describe the process of sampling using a sheaf morphism.

\begin{definition}
\label{df:sampling_morphism}
Suppose that $\shf{F}$ is a sheaf on an abstract simplicial complex $X$, and that $\shf{S}$ is a $V$-sampling sheaf on $X$ supported on a closed subcomplex $Y$.  A \emph{sampling morphism} (or \emph{sampling}) of $\shf{F}$ is a morphism $s:\shf{F}\to \shf{S}$ that is surjective on every stalk.  
\end{definition}

\begin{example}
The diagram below shows a morphism (vertical arrows) between two sheaves, namely the sheaf of continuous functions defined in Example \ref{eg:continuous_sheaf1} (top row) and the sampling sheaf defined in Example \ref{eg:sample_sheaf} (bottom row):
\begin{equation*}
\xymatrix{
C((-2,0),\mathbb{R}) \ar[r]\ar[d]^{e_{-1}} & C((-1,0),\mathbb{R})\ar[d] & C((-1,1),\mathbb{R}) \ar[d]^{e_0}\ar[r] \ar[l] & C((0,1),\mathbb{R})\ar[d]\\
\mathbb{R} \ar[r] & 0 & \mathbb{R} \ar[r]\ar[l] & 0 \\
}
\end{equation*}
In the diagram, $e_x$ represents the operation of evaluating a continuous function at $x$.  As in Section \ref{sec:unifying}, this sampling morphism takes a continuous function $f\in C((-2,1),\mathbb{R})$ to a vector $(f(-1),f(0))$.
\end{example}

In algebraic topology, special emphasis is placed on \emph{sequences} of maps of the form 
\begin{equation*}
\xymatrix{
\dotsb \ar[r] & A_1 \ar[r]^{m_1} & A_2 \ar[r]^{m_2} & A_3 \ar[r]^{m_3} & A_4 \ar[r]^{m_4} & \dotsb,
}
\end{equation*}
where the $A_k$ are vector spaces and the $m_k$ are linear maps.  We will denote this sequence by $(A_\bullet,m_\bullet)$.  For instance, the cochain complex described in the previous section is a sequence of vector spaces.  A linear map satisfies the dimension theorem, which relates the size of its kernel, cokernel, and image.  In some sequences, the dimension theorem is extremely useful -- these are the exact sequences.

\begin{definition}
A sequence $(A_\bullet,m_\bullet)$ of vector spaces is called \emph{exact} if $\ker m_k = \image m_{k-1}$.
\end{definition}

Via the dimension theorem, exact sequences can encode information about linear maps, namely
\begin{enumerate} 
\item $\xymatrix{0 \to A \ar[r]^m & B}$ is exact if and only if $m$ is injective,
\item $\xymatrix{A \ar[r]^m & B \to 0}$ is exact if and only if $m$ is surjective, and
\item $\xymatrix{0 \to A \ar[r]^m & B \to }$ is exact if and only if $m$ is an isomorphism.
\end{enumerate}

Observe that the cochain complex $(C^\bullet(X;\shf{S}),d^\bullet)$ is exact if and only if $H^k(X;\shf{S})=0$ for all $k$.

\begin{remark}
Sequences of sheaf morphisms (instead of just vector spaces) are surprisingly powerful, and play an important role in the general theory of sheaves.  However, if the direction of the morphisms is allowed to change across the sequence, like 
\begin{equation*}
\xymatrix{
\shf{A} & \shf{B} \ar[r]\ar[l] & \shf{C},
}
\end{equation*}
the resulting construction can represent all linear, shift-invariant filters \cite{Robinson_GlobalSIP, Robinson_tspbook}.
\end{remark}

\section{The general sampling theorems}
\label{sec:general_sampling}

Given a sampling morphism, we can construct the \emph{ambiguity sheaf} $\shf{A}$ in which the stalk $\shf{A}(a)$ for a face $a\in X$ is given by the kernel of the map $\shf{F}(a)\to \shf{S}(a)$.  If $a\attach b$ is an attachment of faces in $X$, then $\shf{A}(a\attach b)$ is given by $\shf{F}(a\attach b)$ restricted to $\shf{A}(a)$.  This implies that the sequence of sheaves
\begin{equation*}
\xymatrix{
0\to \shf{A} \ar[r] & \shf{F} \ar[r]^{s} & \shf{S} \to 0
}
\end{equation*}
induces short exact sequences of cochain spaces
\begin{equation*}
0 \to C^k(X;\shf{A}) \to C^k(X;\shf{F}) \to C^k(X;\shf{S}) \to 0,
\end{equation*}
one for each $k$.  Together, these sequences of cochain spaces induce a long exact sequence (via the well-known Snake lemma; see \cite{Hatcher_2002} for instance)
\begin{equation*}
0\to H^0(X;\shf{A}) \to H^0(X;\shf{F}) \to H^0(X;\shf{S}) \to H^1(X;\shf{A}) \to \dotsb
\end{equation*}
An immediate consequence is therefore
\begin{corollary} (Sheaf-theoretic Nyquist theorem)
\label{thm:nyquist}
The global sections of $\shf{F}$ are identical with the global sections of $\shf{S}$ if and only if $H^k(X;\shf{A})=0$ for $k=0$ and $1$.
\end{corollary}

The cohomology space $H^0(X;\shf{A})$ characterizes the \emph{ambiguity} in the sampling.  When $H^0(X;\shf{A}$ is nontrivial, there are multiple global sections of $\shf{F}$ that result in the same set of samples.  In contrast, $H^1(X;\shf{A})$ characterizes the \emph{redundancy} of the sampling.  When $H^1(X;\shf{A}$ is nontrivial, then there are sets of samples that correspond to \emph{no} global section of $\shf{F}$.  Optimal sampling therefore consists of identifying minimal closed subcomplexes $Y$ so the resulting ambiguity sheaf $\shf{A}$ has $H^0(X;\shf{A})=H^1(X;\shf{A})=0$.

\begin{remark}
Corollary \ref{thm:nyquist} is also useful for describing boundary value problems for differential equations.  The sheaf $\shf{F}$ can be taken to be a sheaf of solutions to a differential equation \cite{Ehrenpreis_1956}. The sheaf $\shf{S}$ can be taken to have support only at the boundary of the region of interest, and therefore specifies the possible boundary conditions.  In this case, the space of global sections of the ambiguity sheaf $\shf{A}$ consists of all solutions to the differential equation that \emph{also} satisfy the boundary conditions.
\end{remark}

Let us place bounds on the cohomologies of the ambiguity sheaf.  To do so, we construct two new sheaves associated to a given sheaf $\shf{F}$ and a closed subcomplex $Y \subseteq X$.  These new sheaves allow us to study reconstruction from a collection of rich samples.

\begin{definition}
For a closed subcomplex $Y$ of $X$, let $\shf{F}^Y$ be the sheaf whose stalks are the stalks of $\shf{F}$ on $Y$ and zero elsewhere, and whose restrictions are either those of $\shf{F}$ on $Y$ or zero as appropriate.  There is a surjective sheaf morphism $\shf{F}\to \shf{F}^Y$ and an induced ambiguity sheaf $\shf{F}_Y$ which can be constructed in exactly the same way as $\shf{A}$ before.  
\end{definition}

Thus, the dimension of each stalk of $\shf{F}^Y$ is larger than that of any sampling sheaf supported on $Y$, and the dimension of stalks of $\shf{F}_Y$ are therefore as small as or smaller than that of any ambiguity sheaf.  Because global sections are determined by their values at the vertices (Proposition \ref{prop:global_coho}), obtaining rich samples from $\shf{F}^Y$ at all vertices evidently allows reconstruction.  This idea works for all degrees of cohomology, which generalizes the notion of oversampling.

\begin{proposition} (Oversampling theorem)
\label{prop:oversample}
If $X^k$ is the closed subcomplex generated by the $k$-faces of $X$, then $H^k(X^{k+1};\shf{F}_{X^k})=0$.
\end{proposition} 

On the other hand, not taking enough samples leads to an ambiguous reconstruction problem.  This can be detected by the presence of nontrivial global sections of the ambiguity sheaf.
 
\begin{theorem} (Sampling obstruction theorem)
\label{thm:obstruction}
Suppose that $Y$ is a closed subcomplex of $X$ and $s:\shf{F}\to \shf{S}$ is a sampling of sheaves on $X$ supported on $Y$.  If $H^0(X,\shf{F}_Y) \not= 0$, then the induced map $H^0(X;\shf{F})\to H^0(X;\shf{S})$ is not injective.
\end{theorem}
Succinctly, $H^0(X,\shf{F}_Y)$ is an obstruction to the recovery of global sections of $\shf{F}$ from its samples.

\subsection{Proofs of the general sampling theorems}

\begin{proof} (of Proposition \ref{prop:oversample})
By direct computation, the $k$-cochains of $\shf{F}_{X^k}$ are 
\begin{eqnarray*}
C^k(X^{k+1};\shf{F}_{X^k})&=&C^k(X^{k+1};\shf{F})/C^k(X^k;\shf{F})\\
&=&\bigoplus_{a\text{ a }k\text{-face of}X} \shf{F}(a) / \bigoplus_{a\text{ a }k\text{-face of}X} \shf{F}(a)\\
&=& 0.
\end{eqnarray*}
\qed
\end{proof}

As an immediate consequence, $H^0(X;\shf{F}_Y)=0$ when $Y$ is the set of vertices of $X$.

\begin{proof} (of Theorem \ref{thm:obstruction})
We begin by constructing the ambiguity sheaf $\shf{A}$ as before so that
\begin{equation*}
\xymatrix{
0\to \shf{A} \to \shf{F} \ar[r]^{s} & \shf{S} \to 0
}
\end{equation*}
is a short exact sequence of sheaves.  Observe that $\shf{S}\to \shf{F}^Y$ can be chosen to be injective, because the stalks of $\shf{S}$ have dimension not more than the dimension of $\shf{F}$ (and hence $\shf{F}^Y$ also).  Thus the induced map $H^0(X;\shf{S})\to H^0(X;\shf{F}^Y)$ is also injective.  Therefore, by a diagram chase on
\begin{equation*}
\xymatrix{
0\to H^0(X;\shf{A}) \ar[r] & H^0(X;\shf{F}) \ar[d]^{\cong} \ar[r]^{s}& H^0(X;\shf{S})\ar[d]\\
0\to H^0(X;\shf{F}_Y) \ar[r] & H^0(X;\shf{F}) \ar[r] & H^0(X;\shf{F}^Y)\\
}
\end{equation*}
we infer that there is a surjection $H^0(X;\shf{A})\to H^0(X;\shf{F}_Y)$.  By hypothesis, this means that $H^0(X;\shf{A})\not= 0$, so in particular $H^0(X;\shf{F})\to H^0(X;\shf{S})$ cannot be injective. \qed
\end{proof}

\section{Examples}

This section shows the unifying power of a sheaf-theoretic approach to sampling, by focusing on three rather different examples.  The examples differ in terms of how ``local'' the reconstruction is; those that are less local show a greater impact of the topology of the base space on reconstruction.   Specifically, we examine
\begin{enumerate}
\item \emph{Bandlimited functions}, in which reconstruction is global.
  Topology strongly impacts the number of samples required: if we
  instead consider bandlimited functions on a compact space, we obtain
  finite Fourier series.  (The sampling \emph{rate} is unchanged,
  however.)
\item \emph{Quantum graphs}, in which reconstruction is somewhat
  local.  Sometimes nontrivial topology in the domain is detected, sometimes not.
\item \emph{Splines}, in which there are only local constraints on the
  functions.  Topology plays almost no role in the reconstruction of
  splines from their samples.
\end{enumerate}

Of course, the case of $PW_B$ is rather well-known -- but we show that it has a sheaf-theoretic interpretation.  In rather stark contrast to the case of $PW_B$ is the vector space consisting of the B-splines associated to a particular knot sequence.  The functions in this space are determined via a locally finite, piecewise polynomial partition of unity.  Since B-splines are determined locally, it makes sense that reconstructing them from local samples is possible.  Importantly, sampling theorems obtained for spaces of B-spline are less sensitive to global topological properties.

Spaces of solutions to linear differential equations are a kind of intermediate between $PW_B$ and the space of B-splines.  While a degree $k$ differential equation defines its solution locally, there are $k$ linearly independent such solutions.  Additionally, the topology of the underlying space on which the differential equation is written impacts the process of reconstruction from samples.  \cite{pesenson_2005,pesenson_2006,RobinsonQGTopo}

The unifying power of sheaf theory means that all of the examples in
this section can be treated in the same way, according to the
following procedure:
\begin{enumerate}
\item \emph{Encode} the function space to be sampled as a sheaf,
\item \emph{Specify} the sampling sheaf and sampling morphism,
\item \emph{Construct} the ambiguity sheaf associated to the sampling morphism, 
\item \emph{Compute} the cohomology of the ambiguity sheaf, and 
\item \emph{Collect} conditions for perfect reconstruction based on this computation.
\end{enumerate}

\subsection{Bandlimited functions}

\label{sec:bandlimited}
\begin{figure}
\begin{center}
\includegraphics[width=3.25in]{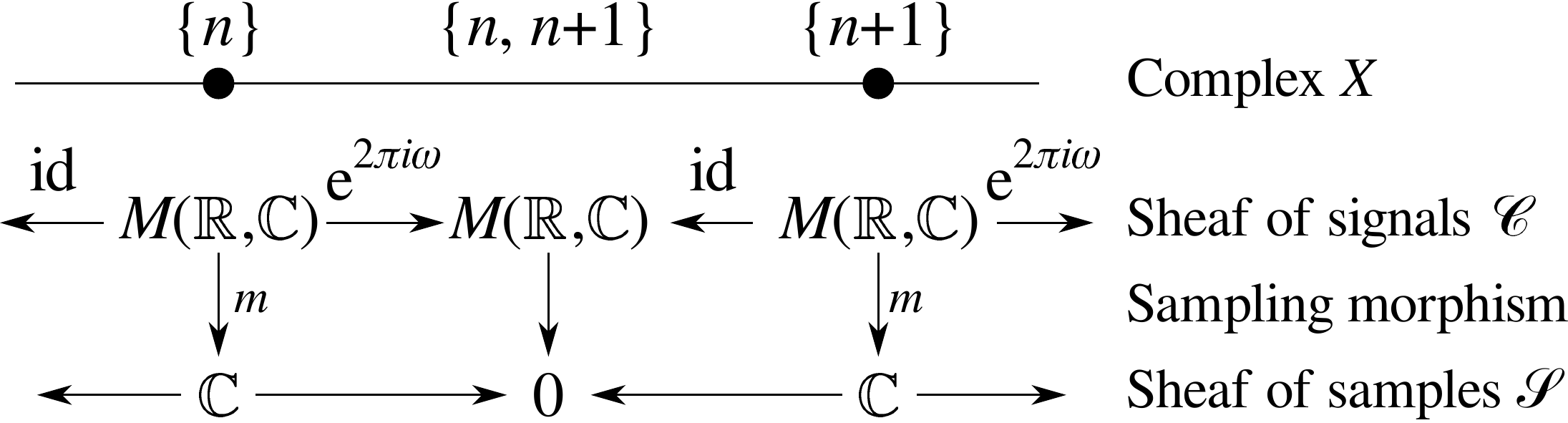}
\caption{The sheaves used in proving the traditional Nyquist theorem}
\label{fig:nyquist_setup}
\end{center}
\end{figure}

In this section, we prove the traditional form of the Nyquist theorem by showing that appropriate bandlimiting is a sufficient condition for $H^0(X;\shf{A})=0$, where $\shf{A}$ is an ambiguity sheaf, and $X$ is an abstract simplicial complex for the real line $\mathbb{R}$.  

We begin by specifying the following 1-dimensional simplicial complex $X$.  Let $X^0=\mathbb{Z}$ and $X^1=\{(n,n+1)\}$.  We construct the sheaf $\shf{C}$ of signals according to their Fourier transforms (see Figure \ref{fig:nyquist_setup}) so that for every simplex, the stalk of $\shf{C}$ is the vector space $M(\mathbb{R},\mathbb{C})$ of complex-valued measures on $\mathbb{R}$.  To relate translation in the spatial domain and the frequency domain, each restriction to the left is chosen to be the identity, and each restriction to the right is chosen to be multiplication by $e^{2\pi i \omega}$.  In essence, $\shf{C}$ is the sheaf of \emph{local Fourier transforms} of functions on $\mathbb{R}$.  Observe that the space of global sections of $\shf{C}$ is therefore just $M(\mathbb{R},\mathbb{C})$.

Construct the sampling sheaf $\shf{S}$ whose stalk on each vertex is $\mathbb{C}$ and each edge stalk is zero.  We construct a sampling morphism by the zero map on each edge, and by the integral
\begin{equation*}
m(f)=\int_{-\infty}^\infty f(\omega) d\omega = f((-\infty,\infty))
\end{equation*}
on each vertex.

Then the ambiguity sheaf $\shf{A}$ has stalks $M(\mathbb{R},\mathbb{C})$ on each edge, and $\{f\in M(\mathbb{R},\mathbb{C}) : m(f)=0\}$ on each vertex $\{n\}$.  

\begin{theorem}(Traditional Nyquist theorem)
Suppose we replace $M(\mathbb{R},\mathbb{C})$ with $M_B(\mathbb{R},\mathbb{C})$, the set of measures whose support is contained in $[-B,B]$.  Then if $B\le 1/2$, the resulting ambiguity sheaf $\shf{A}_B$ has $H^0(X;\shf{A}_B)=0$.  Therefore, each such function can be recovered uniquely from its samples on $\mathbb{Z}$.
\end{theorem}
\begin{proof}
The elements of $H^0(X;\shf{A}_B)$ are given by the measures $f$ supported on $[-B,B]$ for which
\begin{equation*}
\int_{-B}^B f(\omega) e^{2\pi i n \omega} d\omega=(e^{2\pi i n \omega} f)([-B,B])=0
\end{equation*}
for all $n$.  Observe that if $B\le 1/2$, this is precisely the statement that the Fourier series coefficients of $f$ all vanish; hence $f$ must vanish.  This means that the only global section of $\shf{A}_B$ is the zero function.  (Ambiguities can arise if $B>1/2$, because the set of functions $\{e^{-2\pi in\omega}\}_{n\in\mathbb{Z}}$ is then \emph{not} complete.) \qed
\end{proof}

\begin{figure}
\begin{center}
\includegraphics[width=2in]{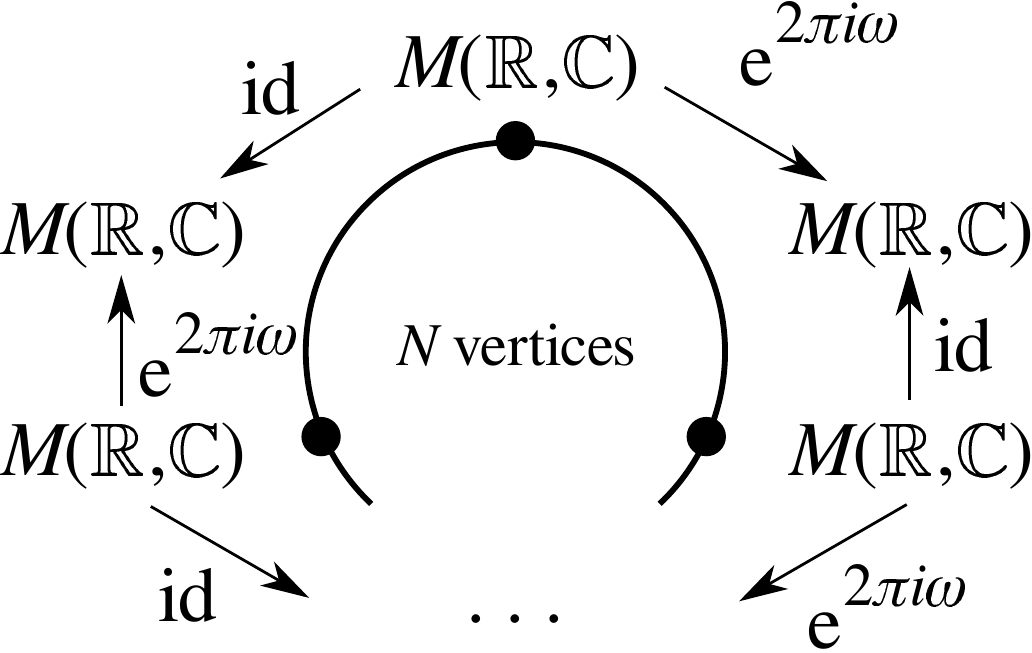}
\caption{The sheaf $\shf{C}$ of local Fourier transforms of functions on a circle with $N$ vertices}
\label{fig:circle_measures}
\end{center}
\end{figure}

Sampling on the circle can be addressed by a related construction of a sheaf $\shf{C}$.  As indicated in Figure \ref{fig:circle_measures}, the stalk over each edge and vertex is still $M(\mathbb{R},\mathbb{C})$.  Again, the restrictions are chosen so that left-going restrictions are identities, and the right-going restrictions consist of multiplying measures by $e^{2\pi i \omega}$.  Intuitively, this means that functions that are local to an edge or a vertex do not reflect any nontrivial topology.  

Since the topology is no longer that of a line, there are some important consequences.  The space of global sections of $\shf{C}$ on the circle is \emph{not} $M(\mathbb{R},\mathbb{C})$, and now depends on the number $N$ of vertices on the circle.\footnote{$N$ must be at least 3 to use an abstract simplicial complex model of the circle.  If $N$ is 1 or 2, one must instead use a CW complex.  This does not change the analysis presented here.} One may conclude from a direct computation that the value of any global section at a vertex must be a measure $f$ satisfying
\begin{equation*}
e^{2\pi i \omega N} f(\omega) = f(\omega).
\end{equation*}
Of course, this means that the support of $f$ must be no larger than the set of fractions $\frac{1}{N}\mathbb{Z}$ because at each $\omega$ either $f(\omega)$ must vanish, or $(e^{2\pi i \omega N}-1)$ must vanish.  Hence a global section describes a function whose (global) Fourier transform is discrete.  If we instead consider the sheaf of bandlimited functions $\shf{C}_B$ by replacing each $M(\mathbb{R},\mathbb{C})$ by $M_B(\mathbb{R},\mathbb{C})$, the resulting space of global sections is finite dimensional.  Perhaps surprisingly, this does not impact the required sampling rate.

\begin{corollary}
If $\shf{C}_B$ is sampled at each each vertex, then a sampling morphism will fail to be injective on global sections if $B>1/2$.   
\end{corollary}

(If $B<1/2$, some sampling morphisms -- such as the zero morphism -- are still not injective.)

\begin{proof}
Merely observe that for a sampling sheaf $\shf{S}$ a necessary condition for injectivity is that 
\begin{eqnarray*}
\dim H^0(\shf{C}_B) &\le& \dim H^0(\shf{S})\\
(2N+1)B &\le& N\\
B &\le& N/(2N+1) < 1/2\\ 
\end{eqnarray*}\qed
\end{proof}

\subsection{Wave propagation (quantum graphs)}
\label{sec:quantum_graph}

A rich source of interesting sheaves arise in the context of differential equations \cite{Ehrenpreis_1956}. Sampling problems are interesting in spaces of solutions to differential equations, because they are restricted enough to have relatively relaxed sampling rates.  Although a differential equation describes a function locally, continuity and boundary conditions allow topology to influence which of these locally defined functions can be extended globally.  

Consider the differential equation
\begin{equation}
\label{eq:helmholtz}
\frac{\partial^2 u}{\partial x^2} + k^2 u = 0
\end{equation}
on the real line, in which $k$ is a complex scalar parameter called the \emph{wavenumber}.  The general solution to this differential equation is the linear combination of two traveling waves, namely
\begin{equation*}
u(x) = c_1 e^{ikx} + c_2 e^{-ikx}.
\end{equation*} 
This means that locally and globally, a given solution is described by an element of $\mathbb{C}^2$.  

Let us now generalize to the case of solutions to \eqref{eq:helmholtz} over a graph $X$.  In order for the space of solutions to be well defined, it is necessary to assign a length to each edge.  We shall write $L(e)$ for the length of edge $e$, which is a positive real number.  

\begin{definition}
The \emph{transmission line sheaf} $\shf{T}$ on $X$ has stalks given by
\begin{enumerate}
\item $\shf{T}(e)=\mathbb{C}^2$ over each edge $e$, and
\item $\shf{T}(v)=\mathbb{C}^{\deg v}$ over each vertex $v$, whose degree is $\deg v$. 
\end{enumerate}
Without loss of generality, assume that each edge $e$ that is attached to a vertex $v$ is assigned a number from $\{1,\dotsc,\deg v\}$.  If $e$ is the $m$-th edge attached to vertex $v$, the restriction $\shf{T}(v\attach e_m)$ is given by
\begin{equation*}
\shf{T}(v\attach e_m)(u_1,\dotsc,u_{\deg v})=
\begin{cases}
\left(u_m,e^{-ikL(e_m)}\left(\frac{2}{\deg v}\sum_{j=1}^{\deg v} u_j-u_m\right)\right)&\text{if }[e_m:v] = 1\\
\left(e^{ikL(e_m)}\left(\frac{2}{\deg v}\sum_{j=1}^{\deg v} u_j-u_m\right),u_m\right)&\text{if }[e_m:v] = -1\\
\end{cases}
\end{equation*}
\end{definition}

Although the definition of a transmission line sheaf is combinatorial, it is an accurate model of the space of solutions on its \emph{geometric realization}.  Under an appropriate definition of the Laplacian operator on the geometric realization of $X$ (such as is given in \cite{Zhang_1993,Baker_2006,Baker_2007}), the sheaf of solutions to \eqref{eq:helmholtz} is isomorphic to a transmission line sheaf. \cite{RobinsonQGTopo} Additionally, there are sensible definitions for bandlimitedness in this geometric realization, which give rise to Shannon-Nyquist theorems. \cite{pesenson_2005,pesenson_2006}  However, our focus will remain combinatorial and topological.  We note that others \cite{pesenson_2008,pesenson_2010} have obtained results in general combinatorial settings, though we will focus on the impact of the topology of $X$ on sampling requirements. 

The easiest sampling result for transmission line sheaves follows immediately from Proposition \ref{prop:global_coho}, namely that a global section of a transmission line sheaf $\shf{T}$ on $X$ is completely specified by its values on the vertices of $X$.

This result is clearly inefficient; merely consider the simplicial complex $X$ for the real line with vertices $X^0=\mathbb{Z}$ and edges $X^1=\{(n,n+1)\}$.  We have already seen that the space of sections of a tranmission line sheaf on $X$ is merely $\mathbb{C}^2$, yet oversampling would have us collect samples at infinitely many vertices!  The missing insight is that the topology of $X$ impacts the global sections of a transmission line sheaf.  

\begin{figure}
\begin{center}
\includegraphics[width=3.5in]{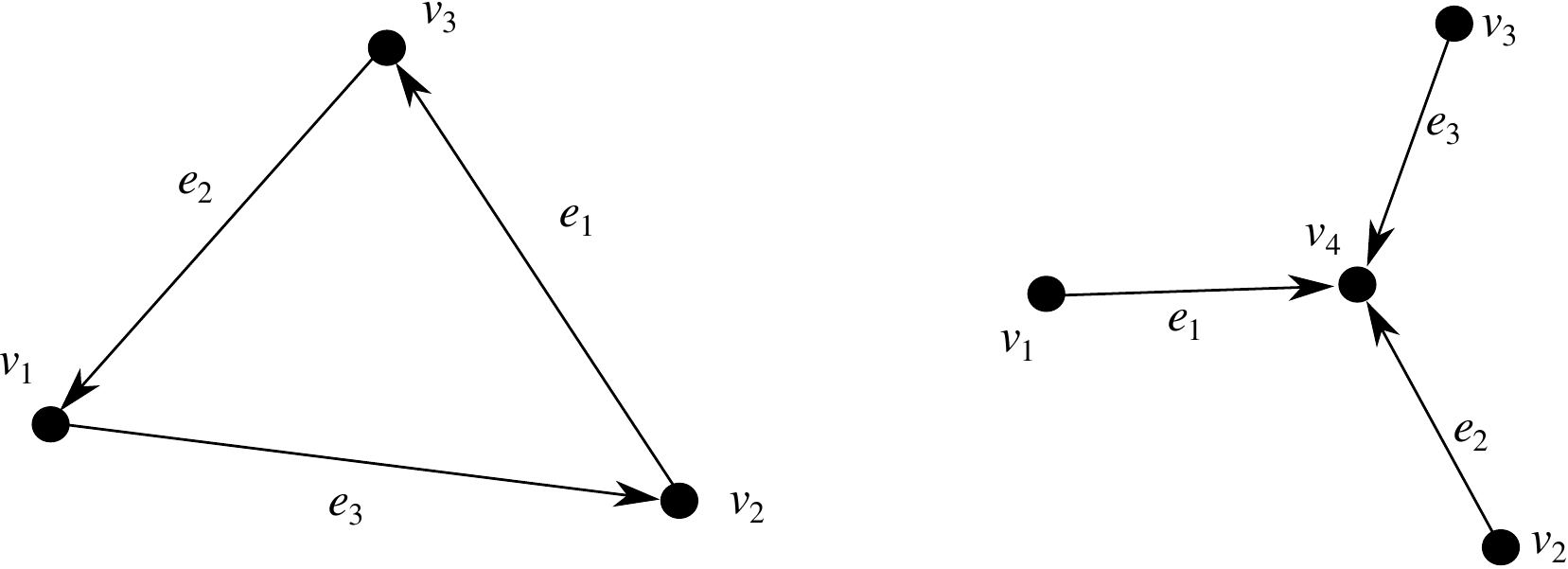}
\caption{The graph $X$ for Example \ref{eg:triangle} (left) and $Y$ for Example \ref{eg:star} (right)}
\label{fig:triangle}
\end{center}
\end{figure}

\begin{figure}
\begin{center}
\includegraphics[width=4in]{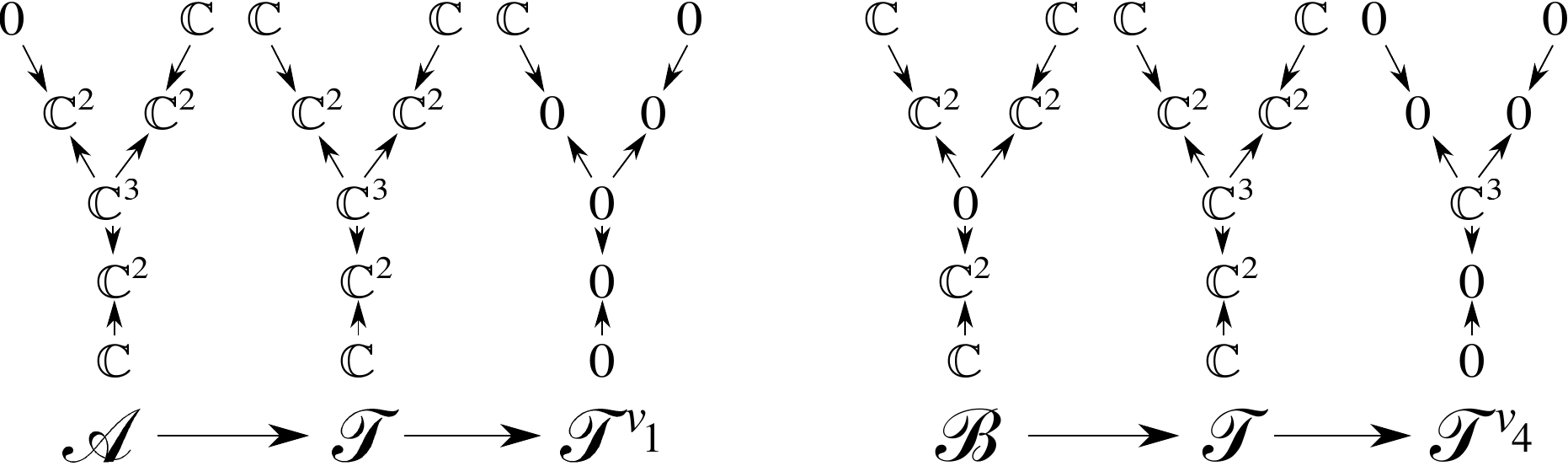}
\caption{Sampling a star at a leaf (left) and at its center (right)}
\label{fig:sample_star}
\end{center}
\end{figure}

Changing the edge length in the simplicial complex model for $\mathbb{R}$ does not change the space of global sections of a transmission line sheaf.  Another siuation in which edge length does not matter is shown in the next example.

\begin{example}
\label{eg:star}
Consider the space $Y$ shown at right in Figure \ref{fig:triangle}.  The coboundary map $d^0$ for a transmission line sheaf over $Y$ is given by
\begin{equation*}
\begin{pmatrix}
-e^{ikL(e_1)}&0&0 &1&0&0\\
-1&0&0 &-\frac{1}{3}e^{-ikL(e_1)}&\frac{2}{3}e^{-ikL(e_1)}&\frac{2}{3}e^{-ikL(e_1)}\\

0&-e^{ikL(e_2)}&0 &0&1&0\\
0&-1&0 &\frac{2}{3}e^{-ikL(e_2)}&-\frac{1}{3}e^{-ikL(e_2)}&\frac{2}{3}e^{-ikL(e_2)}\\

0&0&-e^{ikL(e_3)} &0&0&1\\
0&0&-1 &\frac{2}{3}e^{-ikL(e_3)}&\frac{2}{3}e^{-ikL(e_3)}&-\frac{1}{3}e^{-ikL(e_3)}\\
\end{pmatrix},
\end{equation*}
which has rank 5 -- in other words $\dim H^0(Y;\shf{T})=1$ for any transmission sheaf $\shf{T}$.  This means that reconstruction of sections requires at least one dimension of measurements; a lower bound.  If we consider a sampling morphism $\shf{T}\to\shf{T}^Z$ for some set of vertices $Z$, this induces an injective map on global sections.  To see this, consider sampling at any one of the leaf nodes or at the center.  

If we sample at a leaf node only, as shown in Figure \ref{fig:sample_star} at left, the ambiguity sheaf $\shf{A}$ has a coboundary matrix given by
\begin{equation*}
\begin{pmatrix}
0&0 &1&0&0\\
0&0 &-\frac{1}{3}e^{-ikL(e_1)}&\frac{2}{3}e^{-ikL(e_1)}&\frac{2}{3}e^{-ikL(e_1)}\\

-e^{ikL(e_2)}&0 &0&1&0\\
-1&0 &\frac{2}{3}e^{-ikL(e_2)}&-\frac{1}{3}e^{-ikL(e_2)}&\frac{2}{3}e^{-ikL(e_2)}\\

0&-e^{ikL(e_3)} &0&0&1\\
0&-1 &\frac{2}{3}e^{-ikL(e_3)}&\frac{2}{3}e^{-ikL(e_3)}&-\frac{1}{3}e^{-ikL(e_3)}\\
\end{pmatrix},
\end{equation*}
which has rank 5, implying that the ambiguity sheaf has only trivial global sections.

On the other hand, sampling at the center yields a different ambiguity sheaf $\shf{B}$, whose coboundary matrix is
\begin{equation*}
\begin{pmatrix}
-e^{ikL(e_1)}&0&0\\
-1&0&0\\

0&-e^{ikL(e_2)}&0 \\
0&-1&0\\

0&0&-e^{ikL(e_3)}\\
0&0&-1\\
\end{pmatrix},
\end{equation*}
which also has trivial kernel.  
\end{example}

If we instead consider a different topology, for instance a circle, then edge lengths do have an impact on the global sections of the resulting transmission line sheaf.

\begin{example}
\label{eg:triangle}
Consider the simplicial complex $X$ shown at left in Figure \ref{fig:triangle}, in which the edges are oriented as marked.  Because $X$ has a nontrivial loop, the lengths of the edges impact the space of global sections of a transmission line sheaf over $X$.  Specifically, if $\shf{T}$ is a transmission line sheaf, its coboundary $d^0:C^0(X;\shf{T})\to C^1(X;\shf{T})$ is given by
\begin{equation*}
\begin{pmatrix}
0&0&-e^{ikL(e_1)}&0&1&0 \\
0&0&0&-1&0&e^{-ikL(e_1)} \\
1&0&0&0&-e^{ikL(e_2)}&0\\
0&e^{-ikL(e_2)}&0&0&0&-1\\
-e^{ikL(e_3)}&0&1&0&0&0 \\
0&-1&0&e^{-ikL(e_3)}&0&0 \\
\end{pmatrix}.
\end{equation*}
This matrix has full rank unless $e^{-ik\left(L(e_1)+L(e_2)+L(e_3)\right)}=1$, a condition called \emph{resonance}.  Therefore, the space of global sections of $\shf{T}$ has dimension
\begin{equation*}
\dim H^0(X;\shf{T}) = \begin{cases}
2 & \text{if } k \left(L(e_1)+L(e_2)+L(e_3)\right) \in 2\pi\mathbb{Z}\\
0 & \text{otherwise}
\end{cases},
\end{equation*}
and an easy calculation shows that sampling at any one of the vertices results in an injective map on global sections.
\end{example}

Based on the previous examples, a sound procedure is to consider the dimension of the space of global sections of $\shf{T}$ to be a lower bound on how much information is to be obtained through sampling.  (Clearly, this may not be enough in some situations, especially if the sampling morphisms are not injective on stalks.)  As described in \cite{RobinsonQGTopo}, a general lower bound on the dimension of $H^0(\shf{T})$ is $n+1$, where $n$ is the number of resonant loops.  (A tighter lower bound exists, but its expression is complicated by the presence of degree 1 vertices.) Therefore, topology plays an important role in acquiring enough information to recover global sections of $\shf{T}$ from samples.

\subsection{Polynomial splines}
\label{sec:spline_knots}
Section \ref{sec:bandlimited} showed how limiting the support of the Fourier transform of a function permitted it to be reconstructed by its values at a discrete subset.  Because of the Paley-Weiner theorem, the smoothness of a function is reflected in the decay of its Fourier transform.  On the other hand, Section \ref{sec:quantum_graph} showed that applying smoothness constraints directly to the function also enables perfect recovery.  This suggests that as we consider smoother functions, we can reconstruct them from more widely spaced samples.  

In this section, we consider sampling from polynomial splines, which are functions whose smoothness is explicitly controlled.  A $C^k$ degree $n$ polynomial spline has $k$ continuous derivatives and is constructed piecewise from degree $n$ polynomial segments (see Figure \ref{fig:piecepoly}).  Because of this, a polynomial spline is infinitely differentiable on all of its domain except at a discrete set of \emph{knot points}, where it has $k$ continuous derivatives.

\begin{figure}
\begin{center}
\includegraphics[width=2.5in]{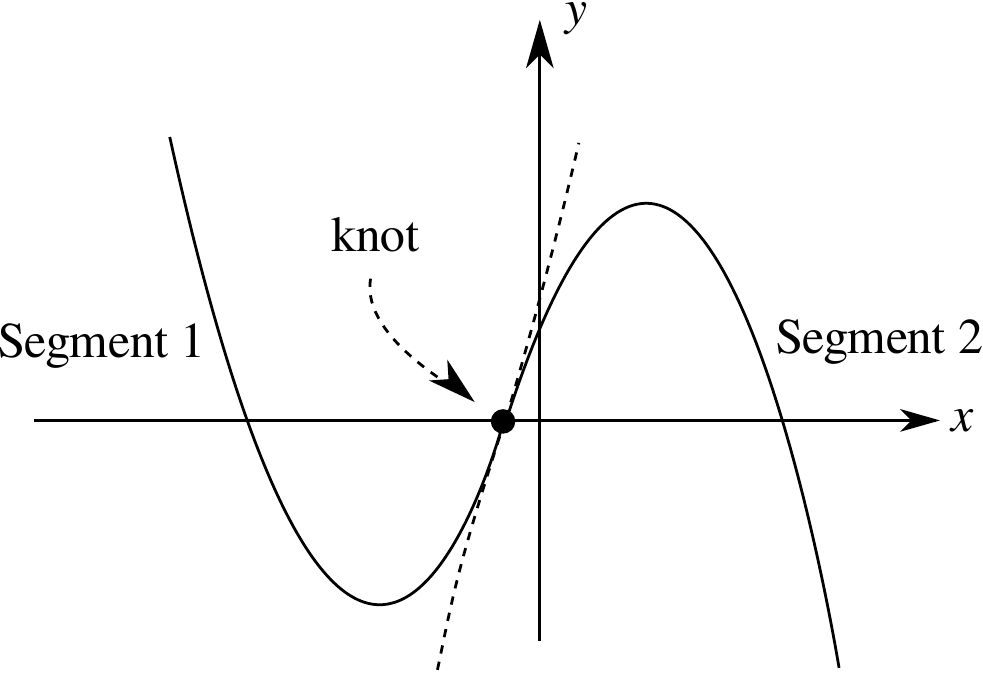}
\caption{A polynomial spline with two quadratic segments, joined at a knot with continuous first derivatives}
\label{fig:piecepoly}
\end{center}
\end{figure}

For instance, consider a degree $n$ polynomial spline that has two knots: one at $0$ and one at $L$.  Require it to have $C^{n-1}$ smoothness across its three segments: $(-\infty,0)$, $(0,L)$, and $(L,\infty)$.  To obtain $n-1$ continuous derivatives at $x=0$, such a spline should have the form
\begin{equation*}
f(x)=\begin{cases}
a_{n}^- x^n + \sum_{k=0}^{n-1} a_k x^k&\text{for } x \le 0\\
a_{n}^+ x^n + \sum_{k=0}^{n-1} a_k x^k&\text{for } 0\le x \le L \\
\end{cases}
\end{equation*}
In a similar way, to obtain $n-1$ continuous derivatives at $x=L$, the spline should be of the form
\begin{equation*}
f(x)=\begin{cases}
b_{n}^- (x-L)^n + \sum_{k=0}^{n-1} b_k (x-L)^k&\text{for } 0 \le x \le L\\
b_{n}^+ (x-L)^n + \sum_{k=0}^{n-1} b_k (x-L)^k&\text{for } x \ge L \\
\end{cases}
\end{equation*}
But clearly on $0 \le x \le L$, these two definitions should agree so that $f$ is a well-defined function.  This means that for all $x$,
\begin{eqnarray*}
a_{n}^+ x^n + \sum_{k=0}^{n-1} a_k x^k &=&b_{n}^- (x-L)^n + \sum_{k=0}^{n-1} b_k (x-L)^k\\
&=&\sum_{i=0}^{n-1} b_i \sum_{k=0}^i \binom{i}{k} x^k (-L)^{i-k} + b_n^- \sum_{k=0}^n \binom{n}{k}x^k(-L)^{n-k}\\
&=&\sum_{k=0}^{n-1} x^k \sum_{i=k}^{n-1}\binom{i}{k}(-L)^{i-k}b_i + \sum_{k=0}^n x^k \binom{n}{k}(-L)^{n-k}b_n^-\\
&=&\sum_{k=0}^{n-1}x^k\left(\sum_{i=k}^{n-1}\binom{i}{k}(-L)^{i-k}b_i + \binom{n}{k}(-L)^{n-k}b_n^-\right)+b_n^-x^n.
\end{eqnarray*}
By linear independence, this means that
\begin{eqnarray*}
a_n^+&=&b_n^-\\
a_k&=&\sum_{i=k}^{n-1}\binom{i}{k}(-L)^{i-k}b_i + \binom{n}{k}(-L)^{n-k}b_n^-\\
\end{eqnarray*}
Notice that if the $a$ variables are given, then this is a triangular system for the $b$ variables.

Using this computation, we can define $\shf{PS}^n$, the sheaf of $n$-degree polynomial splines on $\mathbb{R}$ with knots at each of the integers.  This sheaf is built on the simplicial complex $X$ for $\mathbb{R}$, whose vertices are $X^0=\mathbb{Z}$ and whose edges are $X^1=\{(m,m+1)\}$.  Because a degree $n$ spline at any given knot can be defined by $n+2$ real values on the two segments adjacent to that knot, we assign 
\begin{equation*}
\shf{PS}^n(\{m\}) = \mathbb{R}^{n+2}
\end{equation*}
for each $m\in\mathbb{Z}$.  Specifically, we will think of these as defining $(a_0,\dotsc,a_{n-1},a_n^-,a_n^+)$ in our calculation above.  The spline on each segment is merely a degree $n$ polynomial, so that
\begin{equation*}
\shf{PS}^n(\{(m,m+1)\}) = \mathbb{R}^{n+1}.
\end{equation*}
For each knot, there are two restriction maps: one to the left and one to the right.  They are given by
\begin{equation*}
\shf{PS}^n(\{m\}\attach\{m,m+1\})= \begin{pmatrix}
1& \dotsb  &0 &0 &0\\
\vdots&         &\vdots  &\vdots  &\vdots\\
0& \dotsb  &1 &0 &0\\
0& \dotsb  &0 &0 &1\\
\end{pmatrix}
\end{equation*}
and
\begin{equation*}
\shf{PS}^n(\{m+1\}\attach\{m,m+1\})= \begin{pmatrix}
1&-L & L^2 & -L^3 & \dotsb & (-L)^n& 0\\
 & 1 & -2L & 3L^2 & \dotsb & \binom{n}{1}(-L)^{n-1}& 0\\
 &   &   1 & -3L  & \dotsb & \binom{n}{2}(-L)^{n-2}& 0\\
 &   &     &   1  & \dotsb & \binom{n}{3}(-L)^{n-3}& 0\\
 &   &     &      & & \vdots& \vdots \\
 &   &     &      & & \binom{n}{k}(-L)^{n-k}& 0 \\
 &   &     &      & & \vdots& \vdots \\
 &   &     &      & & 1& 0\\
\end{pmatrix}
\end{equation*}

\begin{remark}
Observe that $L=1$, $\shf{PS}^1$ reduces to the sheaf $\shf{PL}$ given in Example \ref{eg:pl_def}, for the special case of the graph being a line (so all vertices have degree 2).
\end{remark}

\begin{figure}
\begin{center}
\includegraphics[width=2in]{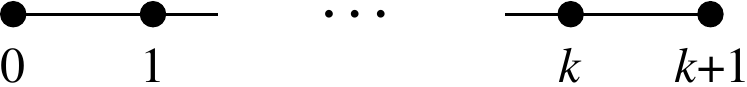}
\caption{The simplicial complex used in Lemma \ref{lem:ps_global}}
\label{fig:ps_global}
\end{center}
\end{figure}

\begin{lemma}
\label{lem:ps_global}
Consider $\shf{PS}^n$ on the simplicial complex shown in Figure \ref{fig:ps_global}, which has $k+2$ vertices and and $k+1$ edges.  The sheaf has nontrivial global sections that vanish at the endpoints if and only if $k>n+1$.  If $k \le n+1$, then the only global section which vanishes at both endpoints is the zero section.
\end{lemma}
\begin{proof}
Consider the ambiguity sheaf $\shf{A}$ associated to the sampling morphism $\shf{PS}^n \to (\shf{PS}^n)^Y$ where $Y$ consists of the two endpoints.  Observe that the global sections of $\shf{A}$ correspond to global sections of $\shf{PS}^n$ that vanish at the endpints, so the lemma follows by reasoning about $H^0(\shf{A})$.  The matrix for the coboundary map $d^0:C^0(\shf{A})\to C^1(\shf{A})$ has a block structure
\begin{equation*}
\begin{pmatrix}
A & 0 & & 0 \\
B & A & \dotsb & 0 \\
0 & B & & 0\\
  &   & & \vdots  \\
& & & A\\
 &  & & B\\
\end{pmatrix},
\end{equation*}
where the $(n+1)\times(n+2)$ blocks are 
\begin{equation*}
A=\begin{pmatrix}
1& \dotsb  &0 &0 &0\\
\vdots&         &\vdots  &\vdots  &\vdots\\
0& \dotsb  &1 &0 &0\\
0& \dotsb  &0 &0 &1\\
\end{pmatrix},\;
B=
\begin{pmatrix}
-1& \dotsb  &* &* &0\\
\vdots&         &\vdots  &\vdots  &\vdots\\
0& \dotsb  &-1 &* &0\\
0& \dotsb  &0 &-1 &0\\
\end{pmatrix}.
\end{equation*}
Clearly both such blocks are of full rank.  Thus the coboundary matrix has a nontrivial kernel whenever it has more rows than columns:
\begin{eqnarray*}
(n+2)k &>& (n+1)(k+1)\\
nk + 1k &>& nk + n + k + 1\\
k &>& n+1\\
\end{eqnarray*}
as desired. \qed
\end{proof}

This lemma implies that unambiguous reconstruction from samples is possible provided the gaps between samples are small enough.  Increased smoothness allows the gaps to larger without inhibiting reconstruction.  Because of this, it is convenient to define distances between vertices.

\begin{definition}
On a graph $G$, define the \emph{edge distance} between two vertices $v,w$ to be 
\begin{equation*}
\ed(v,w)=\begin{cases}
\min_p\{\text{\# edges in }p\text{ such that }p\text{ is a }\\\text{PL-continuous path from }v\to w\}\\
\infty\text{ if no such path exists}
\end{cases}
\end{equation*}
From this, the maximal distance to a vertex set $Y$ is 
\begin{equation*}
\med(Y)=\max_{x\in X^0} \{\min_{y\in Y}\; \ed(x,y)\}.
\end{equation*}
\end{definition}

\begin{corollary}
Suppose that $Y\subseteq \mathbb{Z}$, which we take to be a subset of vertices of $X$.  If $\med Y \le n+2$, then the sampling morphism $\shf{PS}^n \to (\shf{PS}^n)^Y$ induces an injective linear map on global sections.
\end{corollary}

Observe that if we instead consider sampling of polynomial splines on the circle, very little changes.  In particular, the proof of Lemma \ref{lem:ps_global} doesn't change at all.  Indeed, we can change the context from polynomial splines on the real line to piecewise linear functions $\shf{PL}$ on a graph.  We'll focus on the special case of a sampling morphism $s:\shf{PL}\to \shf{PL}^Y$ where $Y$ is a subset of the vertices of $X$.  Excluding one or two vertices from $Y$ does not prevent reconstruction in this case, because the samples include information about slopes along adjacent edges.

\begin{figure}
\begin{center}
\includegraphics[width=1.75in]{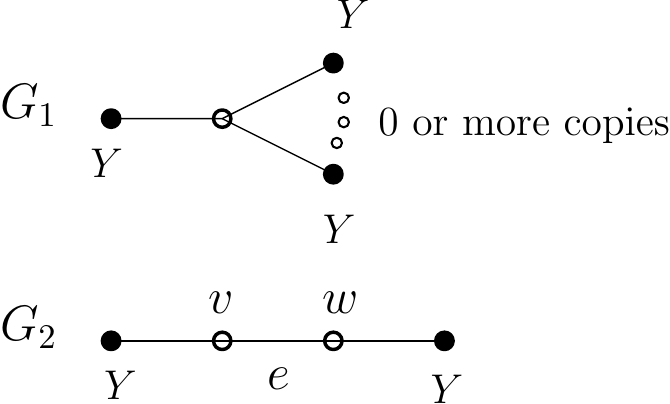}
\includegraphics[width=1.5in]{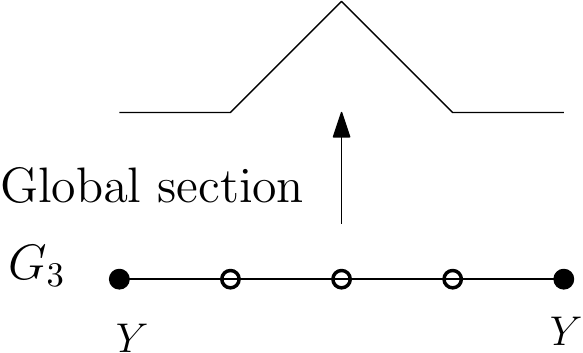}
\caption{Graphs $G_1$, and $G_2$ (left) and $G_3$ (right) for Lemma \ref{lem:pl_cases}.  Filled vertices represent elements of $Y$, empty ones are in the complement of $Y$.}
\label{fig:pl_cases}
\end{center}
\end{figure}

\begin{lemma}
\label{lem:pl_cases}
Consider $\shf{PL}_Y$, the subsheaf of $\shf{PL}$ whose sections vanish on a vertex set $Y$ and the graphs $G_1$, $G_2$, and $G_3$ as shown in Figure \ref{fig:pl_cases}.  There are no nontrivial sections of $\shf{PL}_Y$ on $G_1$ and $G_2$, but there are nontrivial sections of $\shf{PL}_Y$ on $G_3$.
\end{lemma}
\begin{proof}
If a section of $\shf{PL}$ vanishes at a vertex $x$ with degree $n$, this means that the value of the section there is an $(n+1)$-dimensional zero vector.  The value of the section on every edge adjacent to $x$ is then the $2$-dimensional zero vector.  Since the dimensions in each stalk of $\shf{PL}$ represent the value of the piecewise linear function and its slopes, linear extrapolation to the center vertex in $G_1$ implies that its value is zero too.

A similar idea applies in the case of $G_2$.  The stalk at $v$ has dimension 3.  Any section at $v$ that extends to the left must actually lie in the subspace spanned by $(0,0,1)$ (coordinates represent the value, left slope, right slope respectively).  In the same way, any section at $w$ that extends to the right must lie in the subspace spanned by $(0,1,0)$.  Any global section must extend to $e$, which must therefore have zero slope and zero value.

Finally $G_3$ has nontrivial global sections, spanned by the one shown in Figure \ref{fig:pl_cases}.\qed
\end{proof}

\begin{proposition}(Unambiguous sampling)
Consider the sheaf $\shf{PL}$ on a graph $X$ and $Y\subseteq X^0$.  Then $H^0(X;\shf{F}_Y)=0$ if and only if $\med(Y) \le 1$.
\end{proposition}

\begin{figure}
\begin{center}
\includegraphics[width=2.5in]{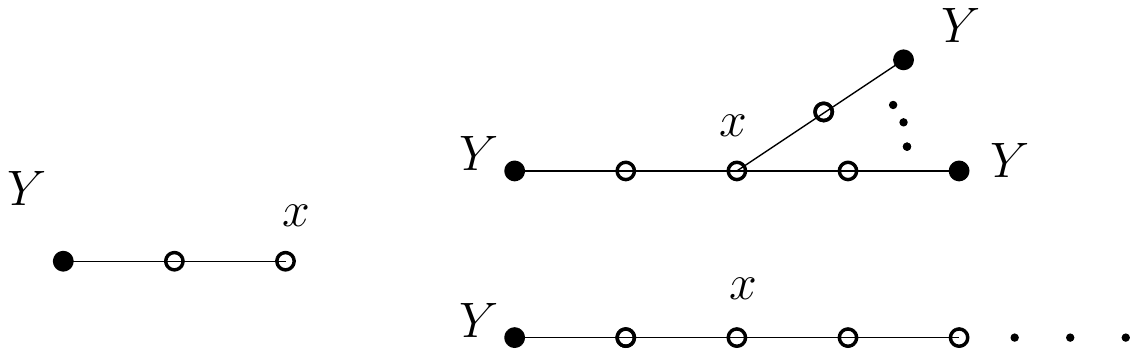}
\caption{The three families of subgraphs that arise when $\med(Y)>1$.  Filled vertices represent elements of $Y$, empty ones are in the complement of $Y$.}
\label{fig:pl_sampling}
\end{center}
\end{figure}

\begin{proof}
($\Leftarrow$) Suppose that $x\in X^0 \backslash Y$ is a vertex not in $Y$.  Then there exists a path with one edge connecting it to $Y$.  Whence we are in the case of $G_1$ of Lemma \ref{lem:pl_cases}, so any section at $x$ must vanish.

($\Rightarrow$) By contradiction.  Assume $\med(Y)>1$.  Without loss of generality, consider $x\in X^0\backslash Y$, whose distance to $Y$ is exactly 2.  Then one of the subgraphs shown in Figure \ref{fig:pl_sampling} must be present in $X$.  But case $G_3$ of Lemma \ref{lem:pl_cases} makes it each of these has nontrivial sections at $x$, merely looking at sections over the subgraph. \qed
\end{proof}

\begin{proposition}(Non-redundant sampling)
Consider the case of $s:\shf{PL}\to \shf{PL}^Y$.  If $Y=X^0$, then $H^1(X;\shf{A})\not= 0$.  If $Y$ is such that $\med(Y)\le 1$ and $|X^0\backslash Y|+\sum_{y\notin Y} \deg y = 2|X^1|$, then $H^1(X;\shf{A})=0$.
\end{proposition}

\begin{proof}
The stalk of $\shf{A}$ over each edge is $\mathbb{R}^2$, and the stalk over a vertex in $Y$ is trival.  However, the stalk over a vertex of degree $n$ not in $Y$ is $\mathbb{R}^{n+1}$.  Observe that if  $H^0(X;\shf{A})=0$, then $H^1(X;\shf{A})=C^1(X;\shf{A})/C^0(X;\shf{A})$.  Using the degree sum formula in graph theory, we compute that $H^1(X;\shf{A})$ has dimension $2|X^1| - \sum_{y\notin Y} (\deg y + 1)$. \qed
\end{proof}

\section{Conclusions}
This chapter has shown that exact sequences of sheaves are a unifying principle for sampling theory.  These tools provide a general basis for discussing locality, and reveal general, precise conditions under which reconstruction succeeds.  Several sampling problems for bandlimited and non-bandlimited functions were discussed.  The use of sheaves in sampling is essentially unexplored otherwise, and there remain many open questions.  We discuss two such questions here with a relationship to bandlimited functions.  

The first question is rather connected to the existing literature on sampling.  We showed that although the number of samples required for the reconstruction of a bandlimited function may vary, the sampling \emph{rate} necessary appears to be the same.  Specifically, we found that the bandwidth $B$ required for reconstructing functions on the real line and the circle was constrained to be less than $1/2$.  Others (for instance \cite{pesenson_2005,pesenson_2010}) have found that this remains the case for other domains as well.  Therefore, it seems fitting to ask ``Is the $B<1/2$ in the Nyquist Theorem invariant with respect to changes in topology and geometry?''

The second question is a bit more subtle.  If a function is bandlimited, this means that its Fourier transform has bounded support.  However, the Fourier transform is intimately related to the spectrum of the Laplacian operator.  On the other hand, because cohomological obstructions play a role in sampling, we had to consider the cochain complex for sheaves of functions.  These two threads of study are in fact closely related through Hodge theory via the study of Hilbert complexes \cite{Bruning_1992}.  Indeed, if $\shf{F}$ and $\shf{S}$ are sheaves of Hilbert spaces then their cochain complexes are Hilbert complexes, and they have a Hodge decomposition, along with an associated Laplacian operator.  Therefore, there may be a way to discuss the concept of ``bandwidth'' for general sheaves of Hilbert spaces.

\begin{acknowledgement}
This work was partly supported under Federal Contract No. FA9550-09-1-0643.  The author also wishes to thank the editor for the invitation to write this chapter.  Portions of this article appeared in the proceedings of SampTA 2013, published by EURASIP.
\end{acknowledgement}

\bibliographystyle{plain}
\bibliography{samplebook_bib}

\begin{thebibliography}{10}

\bibitem{Baker_2007}
Mathhew Baker and Robert Rumely.
\newblock Harmonic analysis on metrized graphs.
\newblock {\em Canad. J. Math.}, 59(2):225--275, 2007.

\bibitem{Baker_2006}
Matthew Baker and Xander Faber.
\newblock Metrized graphs, {L}aplacian operators, and electrical networks.
\newblock In {\em Quantum graphs and their applications}, pages 15--34, 2006.

\bibitem{benedetto_1990}
J.~Benedetto and W.~Heller.
\newblock Irregular sampling and the theory of frames: {I}.
\newblock {\em Note di Mathematica}, 10(1):103--125, 1990.

\bibitem{Bredon}
Glen Bredon.
\newblock {\em Sheaf theory}.
\newblock Springer, 1997.

\bibitem{Bruning_1992}
J.~Br\"uning and M.~Lesch.
\newblock Hilbert complexes.
\newblock {\em Journal of Functional Analysis}, 108:88--132, 1992.

\bibitem{chazal_2009}
F.~Chazal, D.~Cohen-Steiner, and A.~Lieutier.
\newblock A sampling theory for compact sets in euclidean space.
\newblock {\em Discrete Comput. Geom.}, 41:461--479, 2009.

\bibitem{Curry}
J.~Curry.
\newblock Sheaves, cosheaves and applications, {\tt arxiv:1303.3255}.
\newblock 2013.

\bibitem{GhristCurryRobinson}
J.~Curry, R.~Ghrist, and M.~Robinson.
\newblock Euler calculus and its applications to signals and sensing.
\newblock In Afra Zomorodian, editor, {\em Proceedings of Symposia in Applied
  Mathematics: Advances in Applied and Computational Topology}, 2012.

\bibitem{dragotti_2007}
P.~Dragotti, M.~Vetterli, and T.~Blue.
\newblock Sampling moments and reconstructing signals of finite rate of
  innovation: {Shannon} meets {Strang-–Fix}.
\newblock {\em {IEEE} Trans. Sig. Proc.}, 55(5), May 2007.

\bibitem{ebata_2006}
M.~Ebata, M.~Eguchi, S.~Koizumi, and K.~Kumahara.
\newblock Analogues of sampling theorems for some homogeneous spaces.
\newblock {\em Hiroshima Math. J.}, 36:125--140, 2006.

\bibitem{Ehrenpreis_1956}
L.~Ehrenpreis.
\newblock Sheaves and differential equations.
\newblock {\em Proc. AMS}, 7(6):1131--1138, December 1956.

\bibitem{feichtinger_2011}
H.~Feichtinger and I.~Pesenson.
\newblock A reconstruction method for band-limited signals on the hyperbolic
  plane.
\newblock {\em Sampl. Theory Signal Image Process.}, 4(2):107--119, 2005.

\bibitem{feichtinger_1994}
H.G. Feichtinger and K.~Gr{\"o}chenig.
\newblock Theory and practice of irregular sampling.
\newblock {\em Wavelets: mathematics and applications}, pages 305--363, 1994.

\bibitem{ghrist_2011}
R.~Ghrist and Y.~Hiraoka.
\newblock Applications of sheaf cohomology and exact sequences to network
  coding.
\newblock {\em preprint}, 2011.

\bibitem{Godement_1958}
R.~Godement.
\newblock {\em Topologie algebrique et th\'eorie des faisceaux}.
\newblock Herman, Paris, 1958.

\bibitem{groechening_1992}
K.~Gr\"ochening.
\newblock Reconstruction algorithms in irregular sampling.
\newblock {\em Mathematics of Computation}, 59(199):181--194, July 1992.

\bibitem{Hatcher_2002}
A.~Hatcher.
\newblock {\em Algebraic Topology}.
\newblock Cambridge University Press, 2002.

\bibitem{Hubbard}
John~H. Hubbard.
\newblock {\em Teichm\"uller Theory, volume 1.}
\newblock Matrix Editions, 2006.

\bibitem{Iverson}
B.~Iverson.
\newblock {\em Cohomology of Sheaves}.
\newblock Aarhus universitet, Matematisk institut, 1984.

\bibitem{Lilius_1993}
J.~Lilius.
\newblock Sheaf semantics for {Petri} nets.
\newblock Technical report, Helsinki University of Technology, Digital Systems
  Laboratory, 1993.

\bibitem{NiySmaWeiHom}
P.~Niyogi, S.~Smale, and S.~Weinberger.
\newblock Finding the homology of submanifolds with high confidence from random
  samples.
\newblock In R.~Pollack, J.~Pach, and J.~E. Goodman, editors, {\em Twentieth
  Anniversary Volume}, pages 1--23. Springer New York, 2009.

\bibitem{pesenson_2001}
I.~Pesenson.
\newblock Sampling of band-limited vectors.
\newblock {\em Journal of Fourier Analysis and Applications}, 7(1):93--100,
  2001.

\bibitem{pesenson_2005}
I.~Pesenson.
\newblock Band limited functions on quantum graphs.
\newblock {\em Proceedings of the American Mathematical Society},
  133(12):3647--3656, 2005.

\bibitem{pesenson_2006}
I.~Pesenson.
\newblock Analysis of band-limited functions on quantum graphs.
\newblock {\em Applied and Computational Harmonic Analysis}, 21(2):230--244,
  2006.

\bibitem{pesenson_2008}
I.~Pesenson.
\newblock Sampling in {Paley-Wiener} spaces on combinatorial graphs.
\newblock {\em Transactions of the American Mathematical Society},
  360(10):5603, 2008.

\bibitem{pesenson_2010}
I.Z. Pesenson and M.Z. Pesenson.
\newblock Sampling, filtering and sparse approximations on combinatorial
  graphs.
\newblock {\em Journal of Fourier Analysis and Applications}, 16(6):921--942,
  2010.

\bibitem{RobinsonQGTopo}
M.~Robinson.
\newblock Inverse problems in geometric graphs using internal measurements,
  {\tt arxiv:1008.2933}.
\newblock 2010.

\bibitem{RobinsonLogic}
M.~Robinson.
\newblock Asynchronous logic circuits and sheaf obstructions.
\newblock {\em Electronic Notes in Theoretical Computer Science}, pages
  159--177, 2012.

\bibitem{Robinson_GlobalSIP}
M.~Robinson.
\newblock Understanding networks and their behaviors using sheaf theory.
\newblock In {\em GlobalSIP}, 2013.

\bibitem{Robinson_tspbook}
M.~Robinson.
\newblock {\em Topological Signal Processing}.
\newblock Springer, 2014.

\bibitem{schuster_2004}
A.~Schuster and D.~Varolin.
\newblock Interpolation and sampling for generalized {Bergman} spaces on finite
  {Riemann} surfaces.
\newblock {\em Revista Matem{\'a}tica Iberoamericana}, 24(2):499--530, 2008.

\bibitem{Shepard_1980}
Allen Shepard.
\newblock {\em A cellular description of the derived category of a stratified
  space}.
\newblock PhD thesis, Brown University, 1980.

\bibitem{Shepard_1985}
Allen Shepard.
\newblock {\em A cellular description of the derived category of a stratified
  space}.
\newblock PhD thesis, Brown University, 1985.

\bibitem{smale_2004}
S.~Smale and D.X. Zhou.
\newblock Shannon sampling and function reconstruction from point values.
\newblock {\em Bulletin of the American Mathematical Society}, 41(3):279--306,
  2004.

\bibitem{unser_1999}
M.~Unser.
\newblock Splines: A perfect fit for signal and image processing.
\newblock {\em Signal Processing Magazine, IEEE}, 16(6):22--38, 1999.

\bibitem{unser_2000}
M.~Unser.
\newblock Sampling--50 years after {Shannon}.
\newblock {\em Proceedings of the IEEE}, 88(4):569--587, 2000.

\bibitem{unser_1998}
M.~Unser and J.~Zerubia.
\newblock A generalized sampling theory without band-limiting constraints.
\newblock {\em Circuits and Systems II: Analog and Digital Signal Processing,
  IEEE Transactions on}, 45(8):959--969, 1998.

\bibitem{vaughan_1991}
R.G. Vaughan, N.L. Scott, and D.R. White.
\newblock The theory of bandpass sampling.
\newblock {\em Signal Processing, IEEE Transactions on}, 39(9):1973--1984,
  1991.

\bibitem{vetterli_2002}
M.~Vetterli, P.~Marziliano, and T.~Blu.
\newblock Sampling signals with finite rate of innovation.
\newblock {\em Signal Processing, IEEE Transactions on}, 50(6):1417--1428,
  2002.

\bibitem{Zhang_1993}
Shouwu Zhang.
\newblock Admissible pairing on a curve.
\newblock {\em Invent. Math.}, 112(1):171--193, 1993.

\end{thebibliography}

\end{document}